\documentclass[11pt, a4paper]{amsart}

\usepackage[ansinew]{inputenc}
\usepackage[all]{xy}
\SelectTips{cm}{}
\usepackage[pagebackref,
  ,pdfborder={0 0 0}
  ,urlcolor=black,a4paper,hypertexnames=false]{hyperref}
\hypersetup{pdfauthor={Pietro Capovilla, Clara L\"oh, Marco Moraschini}
  ,pdftitle=Amenable category and complexity}

\usepackage{amsfonts}
\usepackage{amssymb, amsmath,amsthm, mathtools}
\usepackage{tabularx, hyperref}
\usepackage{enumerate}

\usepackage{tikz-cd}

\makeatletter
\newtheorem*{rep@theorem}{\rep@title}
\newcommand{\newreptheorem}[2]{%
\newenvironment{rep#1}[1]{%
 \def\rep@title{#2 \ref{##1}}%
 \begin{rep@theorem}}%
 {\end{rep@theorem}}}
\makeatother

\newtheorem{intro_thm}{Theorem}

\newtheorem{intro_prop}[intro_thm]{Proposition}
\newtheorem{intro_quest}[intro_thm]{Question}

\newtheorem{lemma}{Lemma}[section]
\newtheorem{thm}[lemma]{Theorem}
\newreptheorem{thm}{Theorem}  
\newtheorem{prop}[lemma]{Proposition}
\newreptheorem{prop}{Proposition}  
\newtheorem{cor}[lemma]{Corollary}
\newreptheorem{cor}{Corollary}  

\theoremstyle{definition}
\newtheorem{defi}[lemma]{Definition}

\newtheorem{quest}[lemma]{Question}
\newreptheorem{quest}{Question}  

\newtheorem{example}[lemma]{Example}
\newtheorem{rem}[lemma]{Remark}

\theoremstyle{definition}

\makeatletter
\newcommand\norm{\bBigg@{0.8}}

\makeatother
 
 \newcommand{\indnorm}[2][flex]{\csname #1l\endcsname\|#2%
                                 \csname #1r\endcsname\|\mathclose{}}
                                  \newcommand{\indnorml}[4][flex]{\csname #1l\endcsname\|#2%
                                 \csname #1r\endcsname\|_{#3}^{#4}\mathclose{}}
\newcommand{\sv}[2][flex]{\indnorm[#1]{#2}}

\newcommand{\connsum}{\mathbin{\#}}

\DeclareMathOperator{\catop}{cat}
\DeclareMathOperator{\LS}{LS}
\DeclareMathOperator{\Am}{Am}
\DeclareMathOperator{\amcat}{\catop_{\Am}}
\DeclareMathOperator{\lscat}{\catop_{\LS}}
\DeclareMathOperator{\gd}{\textup{gd}}

\DeclareMathOperator{\gcat}{\catop_{\mathcal{G}}}
\DeclareMathOperator{\efcat}{\mbox{efficient $\mathcal{G}$-category}}
\DeclareMathOperator{\lex}{\textup{Lex}}

\DeclareMathOperator{\st}{\textup{st}}
\DeclareMathOperator{\PD}{\textup{PD}}
\DeclareMathOperator{\cd}{cd}

\DeclareMathOperator{\TC}{TC}
\DeclareMathOperator{\comp}{comp}

\DeclareMathOperator{\im}{im}

\DeclareMathOperator{\id}{id}

\newcommand{\fa}[1]{%
  \forall_{#1}\quad}

\newcommand{\qand}{%
  \qquad\text{and}\qquad}

\newcommand{\N}{\ensuremath {\mathbb{N}}}
\newcommand{\R} {\ensuremath {\mathbb{R}}}

\newcommand{\Z} {\ensuremath {\mathbb{Z}}}

\renewcommand{\rho}{\varrho}
\def\phi{\varphi}

\def\geqone{\mathbin{\geq_1}}

\def\args{\;\cdot\;}

\def\connsum{\mathbin{\#}}

\usepackage{color}
\usepackage{pdfcolmk}

\long\def\forget#1{}

\def\longrightarrow{\rightarrow}

\def\widetilde{\tilde}

\makeatletter
\@namedef{subjclassname@2020}{%
  \textup{2020} Mathematics Subject Classification}
\makeatother

\begin{document}

\title{Amenable category and complexity}

\author[]{Pietro Capovilla}
\address{Fakult\"at f\"ur Mathematik, Universit\"at Regensburg, Regensburg, Germany}
\email{ilcapovilla@gmail.com}

\author[]{Clara L\"{o}h}
\address{Fakult\"{a}t f\"{u}r Mathematik, Universit\"{a}t Regensburg, Regensburg, Germany}
\email{clara.loeh@ur.de}

\author[]{Marco Moraschini}
\address{Fakult\"{a}t f\"{u}r Mathematik, Universit\"{a}t Regensburg, Regensburg, Germany}
\email{marco.moraschini@ur.de}

\thanks{}

\date{\today.\ 
  This work was supported by the CRC~1085 \emph{Higher Invariants}
  (Universit\"at Regensburg, funded by the~DFG)}

\begin{abstract}
  Amenable category is a variant of the Lusternik-Schnirelman category,
  based on covers by amenable open subsets. We study the monotonicity
  problem for degree-one maps and amenable category and the relation 
  between amenable category and topological complexity. 
\end{abstract}

\maketitle

\section{Introduction}


In applied algebraic topology, various integer-valued invariants
associated with covers are considered to model, e.g., sensor coverage
problems~\cite{DeSilva:Ghrist} or motion planning problems~\cite{Farber:top:compl:of:motion:planning}.
Prototypical invariants of this type are the Lusternik-Schnirelman
category~$\lscat$ and the topological complexity~$\TC$, as introduced
by Farber~\cite{Farber:top:compl:of:motion:planning}: The Lusternik-Schnirelman category~$\lscat(X)$
of a topological space~$X$ is the minimal number of open and in~$X$
contractible sets needed to cover~$X$; the topological
complexity~$\TC(X)$ is the minimal number of open sets that cover~$X
\times X$ and such that over each member of the cover the path
fibration~$PX \longrightarrow X \times X$ (Section~\ref{subsec:tc})
admits a continuous section.

Relaxing the contractibility condition in the definition of~$\lscat$
to conditions on the allowed images on the level of~$\pi_1$ leads to
generalised notions of category, e.g., to the amenable
category~$\amcat$~\cite{GGH:am} (Section~\ref{subsec:cat}).
The class of amenable groups allows for much bigger flexibility
in the covers, such as using ``circle-shaped'' sets, but is still a class
of groups that is often considered as ``small'' or ``negligible'' in
the context of large-scale topology. 

In the present article, we will focus on amenable category and its
relation to topological complexity. Our leading questions are:

  \begin{intro_quest}[Question~\ref{q:ammono:new}]\label{q:ammono}
  Does the following hold for all oriented closed connected
  manifolds~$M$ and~$N$:
  \[ M \geqone N \Longrightarrow \amcat M \geq \amcat N \text{\;?}
  \]
  Here, we write~$M \geqone N$ if $\dim M = \dim N$ and there
  exists a continuous map~$M\longrightarrow N$ of degree~$\pm 1$.
\end{intro_quest}

\begin{intro_quest}[Question~\ref{q:amtc:new}]\label{q:amtc}
  For which topological spaces~$X$ do we have
  \[ \amcat(X \times X) \leq \TC(X) \text{\;?}
  \]
\end{intro_quest} 

As main tools we will use bounded cohomology and classifying spaces
of families of subgroups.

\subsection{Small amenable category and the fundamental group}

It is known that closed connected manifolds~$M$ with~$\lscat M = 3$
and $\dim M > 2$ have free fundamental
group~\cite[Theorem~1.1]{DKR}. In contrast, for each finitely
presented group~$\Gamma$ there exists an oriented closed connected
$5$-manifold~$M$ with~$\pi_1(M) \cong \Gamma$ and $\lscat M =
4$~\cite[Theorem~1.3]{DKR}. For~$\amcat$ we obtain an analogous picture;
as amenable subsets are a richer class of sets than contractible
(within the ambient space) sets, a shift by~$1$ occurs:

\begin{intro_thm}[small values of~$\amcat$;
    Proposition~\ref{prop:amcat3}, Corollary~\ref{cor:amcat2}]\label{thm:smallamcat}
  Let $\Gamma$ be a finitely presented group.
  \begin{enumerate}
  \item If $n \in \N_{\geq 4}$, then there exists an oriented closed
    connected $n$-man\-i\-fold~$M$ with~$\pi_1(M) \cong \Gamma$ and
    $\amcat M \leq 3$.
  \item If $\Gamma$ is non-amenable, then the following are equivalent:
    \begin{enumerate}
    \item The group~$\Gamma$ is the fundamental group of a graph of
      groups whose vertex (and edge) groups are all amenable.
    \item If $X$ is a CW-complex with~$\pi_1(X) \cong \Gamma$,
      then $\amcat X = 2$.
    \item There exists an oriented closed connected manifold~$M$ that
      satisfies~$\pi_1(M) \cong \Gamma$ and $\amcat M = 2$.
    \end{enumerate}
  \end{enumerate}
\end{intro_thm}

Combining the second part of Theorem~\ref{thm:smallamcat} with known
results on Serre's property~FA, bounded cohomology or $L^2$-Betti
numbers leads to many examples of closed manifolds with amenable
category bigger than~$2$ (Corollary~\ref{cor:amcatnot2}).

\subsection{The monotonicity problem for amenable category}

An interesting open problem about Lusternik-Schnirelmann
category is to understand its behaviour under degree-one maps. More
precisely, Rudyak asked the following
question~\cite{Rudyak:question}:

\begin{quest}[monotonicity problem {\cite[Open Problem~2.48]{CLOT}}]\label{quest:rudyak}
  Does the following hold for all oriented closed connected
  manifolds~$M$ and~$N$:
  \[ M \geqone N \Longrightarrow \lscat M \geq \lscat N \text{\;?}
  \]
\end{quest}

The main motivation behind the previous question is the fact that in
general the domain manifold is ``bigger'' than the target
manifold (Remark~\ref{rem:degree1maps:epi:mono},
Proposition~\ref{prop:degree:sv}). 
Rudyak's monotonicity problem for~$\lscat$ is wide open in full
generality. Several partial -- positive -- results are
known~\cite{Rudyak:question, Rudyak:maps:new, DS:connected1,
  DS:rudyak:20,Dranishnikov:Iwase,rudyaksarkar}. We provide several example classes
in which the corresponding question for~$\amcat$ (Question~\ref{q:ammono}) has a
positive answer: It is straightforward to show that oriented closed
connected surfaces and target manifolds with positive simplicial volume
have this property.
Combining previous computations of~$\amcat$ for $3$-manifolds
and known descriptions of the relation~$\geqone$ in dimension~$3$
shows that Question~\ref{q:ammono} has an affirmative answer in
dimension~$3$ (Theorem~\ref{thm:gen:R:quest:3:mflds}).

Moreover, we explain how bounded cohomology can be used to obtain
further positive results (Section~\ref{sec:degree:one}). An instance of this procedure is:

\begin{intro_thm}[Corollary~\ref{thm:gen:rudyak:conj:intro:new}]\label{thm:gen:rudyak:conj:intro}
  Let $\Gamma$ be the fundamental group of
  an oriented closed connected hyperbolic $k$-manifold of dimension~$k
  \in \{2,3\}$. Then, for
  every~$n \geq 2 k$ there exists an oriented closed connected
  $n$-manifold~$N$ with $\pi_1(N) \cong \Gamma$ such that: For all
  oriented closed connected $n$-manifolds~$M$ we have
  \[ M \geqone N \Longrightarrow \amcat M \geq \amcat N.\]
\end{intro_thm}

\subsection{Amenable category and topological complexity}

The topological complexity~$\TC(X)$ is related to the category of~$X
\times X$ with respect to the diagonal family of subgroups
of~$\pi_1(X\times X)$~\cite{FGLO}. In general, we may not expect a direct
connection between this diagonal category and~$\amcat (X \times X)$.
However, it turns out that Question~\ref{q:amtc} has an affirmative
answer in many cases:

\begin{intro_thm}[Theorem~\ref{thm:examples:new}]\label{thm:examples}
  The following classes of spaces satisfy the estimate
  in Question~\ref{q:amtc}:
  \begin{enumerate}
  \item Spaces with amenable fundamental group;
  \item Spaces of type $B\Gamma$ where $\Gamma$ is a finitely
    generated geometrically finite hyperbolic group;
  \item Spaces of type $B\Gamma$ where $\Gamma=H*H$ is the free square
    of a geometrically finite group $H$;
  \item Manifolds whose fundamental group is the fundamental group of
    a graph of groups whose vertex (and edge) groups are all amenable.
  \end{enumerate}
\end{intro_thm}

\begin{intro_thm}[Corollary~\ref{thm:am:vs:tc:3:mflds:intro:new}]\label{thm:am:vs:tc:3:mflds:intro}
  Let $Y$ be an oriented closed connected $3$-manifold that is a
  connected sum of graph manifolds. Then, for every~$n \geq 6$, there
  exists an oriented closed connected $n$-manifold~$M$ with $\pi_1(M)
  \cong \pi_1(Y)$ and such that
  $$
  \amcat(M \times M) \leq \TC (M) \ .
  $$  
\end{intro_thm}

\subsection{Generalising the class of subgroups}

In fact, most of our methods and results also apply to classes of
groups that are contained in the class of amenable groups and that are
closed under isomorphisms, subgroups, and quotients. In the remainder
of the article, we will thus usually stick to the more general
setup. In particular, in analogy with the corresponding criterion for
the category of the trivial group by Eilenberg and
Ganea~\cite{eilenbergganea}, we will show how these generalised
category invariants can be expressed in terms of classifying spaces of
families of subgroups:

\begin{intro_prop}[Proposition~\ref{prop:introgencat:new}]
  \label{prop:introgencat}
  Let $X$ be a connected CW-complex and let $\widetilde{X}$ be 
  its universal covering.  Let $\Gamma$ be the fundamental
  group of~$X$ and let $F$ be a subgroup family of~$\Gamma$.   
   Then,
  $\catop_F(X)-1$ coincides with the minimal
  integer~$k\in\mathbb{N}_{\geq 0}$ such that the classifying map
  \[
  f_{\widetilde X,\Gamma, F}\colon\widetilde X \rightarrow E_F\Gamma
  \]
  is $\Gamma$-homotopic to a map with values in the $k$-dimensional
  skeleton~$E_F\Gamma^{(k)}$ (of any model of~$E_F\Gamma$).
\end{intro_prop}

This leads to a corresponding lower bound of~$\catop_F$
in terms of Bredon cohomology (Corollary~\ref{cor:lower:bound:bredon}).
For example, this can be applied in the case of hyperbolic
fundamental groups (Example~\ref{exa:hypbredon}).

\subsection*{Acknowledgements}

We are grateful to Kevin Li for interesting discussions
and for pointing out Example~\ref{exa:hypbredon}.

We thank the anonymous referee for the careful 
reading of our manuscript and all the mathematical 
and stylistic suggestions to improve the exposition.

Finally, we would like to thank Dieter Kotschick for
his interest in our work. In particular, we are grateful
to him for pointing out an alternative proof of
Theorem~\ref{thm:gen:R:quest:3:mflds} (see
Remark~\ref{rem:3mfldsdirect}).

\subsection*{Organisation of this article}

We recall basic notions on category and topological complexity
(Section~\ref{sec:cats}) as well as on bounded cohomology
(Section~\ref{sec:bc}). In Section~\ref{sec:cplxs:mflds}, we explain how to
turn category computations for complexes into manifold examples.
Section~\ref{sec:smallamcat} deals with the influence of the
fundamental group on small amenable category and contains a proof of
Theorem~\ref{thm:smallamcat}.  In Section~\ref{sec:degree:one}, we
discuss an extension of Rudyak's conjecture for degree one maps and
$\LS$-category and prove Theorem~\ref{thm:gen:rudyak:conj:intro}.  The
description of categorical invariants in terms of classifying spaces
is given in Section~\ref{sec:catclass}.  Finally,
Section~\ref{sec:examples} contains the proof of
Theorem~\ref{thm:examples} and of
Theorem~\ref{thm:am:vs:tc:3:mflds:intro}.

\section{Category and complexity}\label{sec:cats}

We briefly recall the notions of category and topological complexity.

\subsection{Category}\label{subsec:cat}

We will work with the following generalised version of the
LS-category~\cite{CP}.

\begin{defi}
	\label{defi:cat:spaces}
	Let $\mathcal{A}$ be a class of topological spaces that
        contains a non-empty space and let $X$ be a topological space.
        A subset $U\subseteq X$ is called
        \emph{$\mathcal{A}$-contractible in $X$} if the inclusion map
        $i\colon U \rightarrow X$ factors homotopically through a
        space in $\mathcal{A}$, i.e., there exists a space $A\in
        \mathcal{A}$ and maps $\alpha\colon U\rightarrow A$ and
        $\beta\colon A\rightarrow X$ such that $\beta \circ i$ is
        homotopic to $\alpha$.
	
	The \emph{$\mathcal{A}$-category} of $X$, denoted by
        $\catop_{\mathcal{A}}(X)$, is the minimal~$n\in
        \mathbb{N}_{\geq 1}$ such that there exists an open cover
        $\{U_1,\dots,U_n \}$ of $X$ by 
        $\mathcal{A}$-contractible sets in $X$.
	If such an integer does not exist, we set $\catop_{\mathcal{A}}(X):=\infty$.
\end{defi}

\begin{example}
	\label{example:lscat}
	Let $\mathcal{A}=\{*\}$ be the class of spaces containing just
        the point space. Then $\catop_{\mathcal{A}}(X)$ coincides with
        the usual notion of $\LS$-category~$\lscat(X)$ of a
        topological space $X$~\cite{CLOT}.  Notice that we did not
        normalise the category, whence $\catop_{\mathcal{A}}(X) = 1$
        if and only if $X$ is contractible.
\end{example}

Moreover, there is also an algebraic version of
category~\cite{GGW}. For the rest of this paper, it will be
convenient to stick to the following convention:

\begin{defi}
  A class of groups is an \emph{isq-class} if it is non-empty and
  if it is closed under taking isomorphisms, subgroups, and quotients.
\end{defi}

\begin{defi}
  \label{defi:cat:groups}
  Let $\mathcal G$ be an isq-class of groups and let $X$ be
  a topological space.
	A subset
        $U\subseteq X$ is called \emph{$\mathcal{G}$-contractible in
          $X$} (or, simply, a $\mathcal{G}$-set) if for every $x\in U$
        we have
	\[
	\im \bigl(\pi_1(U\hookrightarrow X,x)\bigr)\in \mathcal{G}.
	\]
	
	The \emph{$\mathcal{G}$-category} of~$X$, denoted by~$\catop_{\mathcal{G}}(X)$,
        is then the minimal~$n\in
        \mathbb{N}_{\geq 1}$ such that there exists an open cover
        $\{U_1,\dots,U_n \}$ of~$X$ by $\mathcal{G}$-sets.  If such an
        integer does not exist, we set
        $\catop_{\mathcal{G}}(X):=\infty$.
\end{defi}

By specialising the previous definition to particular isq-classes of
groups, we get the definitions of $\pi_1$-category~\cite{Fox,
  GG:3mfld} and $\Am$-category~\cite{GGH:am}.

\begin{defi}\label{def:amcat}
  \item
	\begin{itemize}
		\item If $\mathcal{G}$ only contains the trivial
                  subgroup~$\{e\}$, then $\catop_{\pi_1} \coloneqq
                  \catop_{\{e\}}$ is the $\pi_1$-\emph{category}.
		\item If $\mathcal{G} = \Am$ is the family of amenable
                  groups, then $\catop_{\Am}$ is the
                  \emph{amenable category}.
		\end{itemize}
\end{defi}

There is a correspondence between the geometric $\mathcal{A}$-category
and the algebraic $\mathcal{G}$-category as follows.

\begin{prop}[{\cite[Proposition 1]{GGW}}]\label{prop:ag:equal:g}
  Let $\mathcal G$ be an isq-class of groups and let
        $\mathcal{A}_{\mathcal{G}}$ denote the class of all topological
        spaces~$Y$ with $\pi_1(Y,y)\in \mathcal{G}$ for every~$y\in
        Y$.  Then, we have
	\[
	\catop_{\mathcal{A}_\mathcal{G}} =\catop_{\mathcal{G}}.
	\]
\end{prop}

The following result generalises a standard estimate for the LS-category.

\begin{prop}
  \label{prop:cat:products}
  Let $\mathcal G$ be an isq-class of groups and let $X$ and $Y$ be
        path-connected topological spaces such that $X\times Y$ is
        completely normal. Then,
	\[
	\catop_{\mathcal{G}}(X\times Y)\leq \catop_{\mathcal{G}}(X) + \catop_{\mathcal{G}}(Y)-1.
	\]
\end{prop}
\begin{proof}
  One can use the same argument as in the case of LS-category~\cite[Theorem
            1.37]{CLOT} once one notices that disjoint unions of open
        $\mathcal{G}$-sets are again $\mathcal{G}$-sets.
\end{proof}

\begin{rem}
	\label{rem:F-cat:leq:lscat}
	Let $X$ a topological space and let $\mathcal{F}$ and
        $\mathcal{G}$ isq-classes of groups with~$\mathcal F \subset \mathcal G$.
        Then, we have
	\[
	\catop_{\mathcal G}(X)\leq \catop_{\mathcal F}(X)\leq \catop_{\pi_1}(X) \leq \lscat(X).
	\]
        If $X$ is a simplicial complex, then $\lscat (X) \leq \dim X + 1$~{\cite[Theorem 1.7]{CLOT}}.
\end{rem}

\begin{rem}\label{rem:pullbackcat}
  Let $\mathcal G$ be an isq-class of groups and 
  let $f \colon X \longrightarrow Y$ be a continuous map between
  connected topological spaces whose kernel~$\ker \pi_1(f)$ lies in~$\mathcal G$.
  Then pulling back open covers of~$Y$ shows that
  \[ \catop_{\mathcal G} X \leq \catop_{\mathcal G} Y.
  \]
  In particular, this applies if $f$ induces an isomorphism
  on the level of fundamental groups.
\end{rem}

\begin{defi}
  We say that path-connected spaces $X$ and $Y$ are $\pi_1$-\emph{equivalent}, if
  there exist continuous maps (called
  $\pi_1$-\emph{equivalences}) $X \to Y$ and $Y \to
  X$ that induce isomorphisms on the level of fundamental groups.
\end{defi}  

\begin{lemma}\label{prop:cathinv}
  If $\mathcal G$ is an isq-class of groups, then $\gcat$ is an invariant
  of $\pi_1$-equivalences.
\end{lemma}

\begin{proof}
  It is sufficient to apply Remark~\ref{rem:pullbackcat} twice.
\end{proof}

We conclude this section with a lemma that allows us to make parts of
a covering disjoint.

\begin{lemma}[{\cite[Lemma A.4]{CLOT}}]
	\label{lemma:reduce:cardinality:open:covers}
	Let $X$ be a topological space and let $n\in \mathbb{N}$. Let
        $U$ be an open cover of $X$ of multiplicity $n+1$ with a partition of
        unity subordinate to the cover. Then there exist and index set
        $B$ and an open covering
	\[
	\bigl\{V_{i\beta} \bigm| i\in\{1,\dots,n+1\}, \beta \in B \bigr\}
	\]
        of $X$ refining $U$
	such that $V_{i\beta}\cap V_{i\beta'}=\emptyset$ for
        all~$\beta\neq\beta'$. 
\end{lemma}

\subsection{Topological complexity}\label{subsec:tc}

Topological complexity is a categorical invariant introduced by
Farber~\cite{Farber:top:compl:of:motion:planning} that is motivated
by the motion planning problem in robotics.  Roughly speaking, it
measures the minimal number of continuous rules one needs to move a
point in a topological space in an autonomous way.  A \emph{motion
  planning algorithm} is a program that, given a topological space $X$
and a pair of points $(A,B)\in X\times X$, returns a continuous path
in $X$ connecting them. In other words, if $PX$ denotes the space of
paths $[0,1]\rightarrow X$ endowed with the compact-open
topology, a motion planning algorithm is a continuous global section
of the path fibration
\begin{align*}
  p\colon PX & \rightarrow X\times X\\
  \gamma & \mapsto (\gamma(0),\gamma(1) )\ .
\end{align*}

\begin{defi}
	\label{defi:topological:complexity}
	Let $X$ be a path-connected topological space. The
        \emph{topological complexity} of $X$, denoted by $\TC(X)$, is the minimum $n\in
        \mathbb{N}_{\geq 1}$ such that there exists an open cover
        $\{U_1,\dots,U_n \}$ of $X\times X$ with the property that for each~$i \in \{1,\dots,n\}$, 
        there exists a continuous section $s_i:U_i\rightarrow PX$ of the path fibration~$p
        \colon PX \rightarrow X \times X$.
         If such an
        integer does not exist, we set
        $\TC(X):=\infty$.
\end{defi}

We collect some classical bounds of the topological complexity by
means of the LS category~\cite{Farber:top:compl:of:motion:planning}.

\begin{prop}
	\label{prop:bounds:TC}
	Let $X$ be a CW-complex and suppose that $X$ is (locally) compact
	or it has countably many cells. Then the
        following inequalities hold:
	\[
	\lscat(X)\leq \TC(X)\leq\lscat(X\times X)\leq 2\cdot \lscat(X) -1.
	\]
\end{prop}

In order to compare $\TC(X)$ with the more algebraic
$\mathcal{G}$-category, it is convenient to work with families of
subgroups.

\begin{defi}
	\label{defi:subgroup:family}
	Let $\Gamma$ be a group. A \emph{subgroup family} of~$\Gamma$
        is a non-empty set~$F$ of subgroups of~$\Gamma$ that is closed under
        conjugation and under taking subgroups.
\end{defi}

\begin{defi}
	\label{defi:cat:subgroup:family}
	Let $X$ be a topological space and let $\Gamma$ denote its
        fundamental group. Suppose that $F$ is a subgroup family of
        $\Gamma$. We say that an open set $U \subset X$ is
        \emph{$F$-contractible} (or, simply, an $F$-set) if for every~$x\in U$ we have
	\[
	\im\bigl(\pi_1(U\hookrightarrow X,x)\bigr)\in F.
	\] 
	An open cover~$\mathcal{U}$ of~$X$ is called an \emph{$F$-cover}
        if it is made of $F$-sets.  The \emph{$F$-category} of~$X$,
        denoted by~$\catop_F(X)$, is the minimal~$n\in
        \mathbb{N}_{\geq 1} \cup \{\infty\}$ such that $X$ admits an open $F$-cover of
        cardinality~$n$.
\end{defi}

\begin{rem}
	\label{rem:cat:subgroup:family}
        If $\mathcal G$ is an isq-class of groups and $\Gamma$ is a group, then
        we have that
	\[
	\mathcal{G}_{\Gamma}\coloneqq\{H\leq \Gamma \;|\; H\in \mathcal{G} \}
	\]
	is a subgroup family of~$\Gamma$ and that all path-connected
        topological spaces~$X$ with fundamental group~$\Gamma$ satisfy 
	\[
	\catop_{\mathcal{G}_{\Gamma}}(X)=\catop_{\mathcal{G}}(X).
	\]
\end{rem}

Moreover, the definition of $F$-category leads to a variation of the
standard topological complexity~\cite{FGLO}. This $\mathcal
D$-topological complexity provides a lower bound for the standard
topological complexity; for finite aspherical complexes one even has
equality.

\begin{defi}\label{def:D:TC}
  Let $X$ be a path-connected topological space with fundamental group~$\Gamma$ and
  let us consider the family~$\mathcal D$ of subgroups of~$\pi_1(X \times X) \cong
  \Gamma \times \Gamma$ generated by
  \[
  \Delta\coloneqq\{(\gamma,\gamma) \in\Gamma\times \Gamma\;|\; \gamma \in \Gamma \}.
  \]
  We define the $\mathcal{D}$-\emph{topological complexity}
  $\TC^{\mathcal{D}}(X)$ as the $\mathcal{D}$-category of $X \times
  X$.
\end{defi}

\begin{lemma}[{\cite[Proposition~2.4]{FGLO:upper}\cite[Lemma 2.3.2]{FGLO}}]
	\label{lemma:D-top:compl:leq:top:compl}
	\label{lemma:D-top:compl:equals:top:compl}
        Let $X$ be a path-connected topological space.
        \begin{enumerate}
        \item If $X$ is locally path-connected and semi-locally simply connected, then $\TC^{\mathcal{D}}(X)\leq \TC(X)$.
        \item If $X$ is a finite aspherical CW-complex, then
	  $
	  \TC^{\mathcal{D}}(X)=\TC(X).
	  $
        \end{enumerate}
\end{lemma}


\section{Bounded cohomology}\label{sec:bc}

We recall terminology from bounded cohomology, simplicial volume, and
their interaction with amenable covers.

\subsection{Bounded cohomology}

Bounded cohomology of groups was first introduced by
Johnson~\cite{Johnson} and Trauber in the setting of Banach
algebras. Later, Gromov~\cite{vbc} extended this notion from groups to
topological spaces. Moreover, there is a description in terms of
homological algebra available~\cite{ivanov}. 
Given a topological space $X$, we denote by
$(C_\bullet(X), \partial_\bullet)$ and $(C^\bullet(X),
\delta^\bullet)$ the \emph{real} singular chain complex and singular
cochain complex, respectively. We endow $C^\bullet(X)$ with the 
$\ell^\infty$-\emph{norm} as follows: For every $k \in \, \N$, if $f
\in \, C^k(X)$, we define
$$ \lVert f \rVert_\infty \coloneqq
\sup \bigl\{ | f(\sigma) | \bigm| \sigma \mbox{ is a singular $k$-simplex}
\bigr\} \in [0,\infty] \ .
$$ 
A singular cochain $f \in \, C^\bullet(X)$ is
\emph{bounded} if $\sv f_\infty < +\infty$. Since by linearity the
singular coboundary operator $\delta^\bullet$ sends bounded cochains
to bounded cochains, we can restrict $\delta^\bullet$ to the
subcomplex of \emph{bounded cochains}:
$$
C_b^\bullet(X) \coloneqq \{ f \in \, C^\bullet(X) \, | \, \sv f _\infty < +\infty\}\ .
$$

\begin{defi}
  We define the \emph{bounded cohomology} of a topological space $X$
  to be the homology of $(C^\bullet_b(X), \delta^\bullet)$ and we
  denote it by $H^\bullet_b(X)$.
\end{defi}

\begin{rem}
  Since bounded cohomology is a homotopy invariant, given a discrete
  group $\Gamma$, we can define $H^\bullet_b(\Gamma)$ as the (real)
  bounded cohomology of any model of $B\Gamma$.
\end{rem}

\begin{rem}
  Bounded cohomology $H^\bullet_b(X)$ is endowed with a natural
  seminorm induced by $\ell^\infty$: For every $\varphi \in \,
  H^\bullet_b(X)$, we set $$\sv \varphi _\infty \coloneqq \inf \{\sv f
  _\infty \, | \, f \in \, C_b^\bullet(X) \mbox{ is a cocycle
    representing } \varphi\} \ .$$
\end{rem}

Notice that the inclusion of the bounded cochain complex
$C_b^\bullet(X)$ into $C^\bullet(X)$ induces a map 
$$
\comp_X^\bullet \colon H^\bullet_b(X) \to H^\bullet(X)
$$
in cohomology, the so-called \emph{comparison map}.

We recall some results that we will need in the sequel.

\begin{thm}[Gromov's mapping theorem~\cite{vbc}\cite{ivanov}]\label{map:thm}
  If $f \colon X \to Y$ is a continuous map between path-connected
  spaces such that $\pi_1(f)$ is an epimorphism with amenable kernel,
  then
  $$H^\bullet_b(f) \colon H^\bullet_b(Y) \to H^\bullet_b(X)$$ is an isometric isomorphism in all degrees. 
\end{thm}

\begin{cor}\label{cor:same:behaviour:pi1:spaces}
  Let $f \colon M \to N$ be a continuous map between
  oriented closed connected manifolds. Then $H^\bullet_b(f)$ is
  injective, surjective or an isomorphism, respectively, if and only
  if $H^\bullet_b(\pi_1(f))$ is injective, surjective or an
  isomorphism, respectively.
\end{cor}

\subsection{Simplicial volume}

In the case of oriented closed connected manifolds, bounded cohomology
can be used to define a numerical invariant called \emph{simplicial
  volume} \cite{vbc}\cite[Section~7.5]{Frigerio:book} (the original
definition is formulated differently, but in view of the duality principle
these two definitions are equivalent).

\begin{defi}
  Let $M$ be an oriented closed connected $n$-manifold. We define the
  \emph{simplicial volume} of $M$ as follows:
  $$
  \sv M \coloneqq \sup \Big\{\frac{1}{\sv \varphi _\infty} \, \Big| \, \varphi \in \, H^n_b(M) \, , \,   \langle \comp^n_M(\varphi), [M] \rangle = 1 \Big\} \ ,
  $$
  where we set $\sup \emptyset := 0$. Here, $[M] \in \, H_n(M)$
  denotes the real fundamental class of $M$ and $\langle \cdot, \cdot
  \rangle$ the Kronecker product between cohomology and homology.
\end{defi}

\begin{rem}\label{rem:duality:principle}
  Let $M$ be an oriented closed connected $n$-manifold. 
  Since $H^n(M)$ is one-dimensional,
  we have
  $$
  \sv M > 0 \Longleftrightarrow \comp_M^n \mbox{ is surjective} \ .
  $$
\end{rem}

\begin{example}\label{exa:zerosimvol}
  Since bounded cohomology of amenable groups vanishes in all positive
  degree~\cite{Johnson, vbc,ivanov}, Theorem~\ref{map:thm} shows that this is
  also the case for the bounded cohomology of oriented closed
  connected manifolds with amenable fundamental group. Hence, the
  previous remark implies that the simplicial volume of all oriented
  closed connected manifolds with amenable fundamental group (and non-zero dimension)
  is zero.
\end{example}

\begin{example}\label{exa:possimvol}
  In contrast with the previous example, the following manifolds are
  known to have positive simplicial volume:
  \begin{itemize}
  \item oriented closed connected hyperbolic manifolds~\cite{Thurston,
    vbc}; in particular, surfaces of genus~$\geq 2$,
  \item oriented closed connected locally symmetric spaces of
    non-compact type~\cite{Bucher:lss, Lafont-Schmidt},
  \item manifolds with sufficiently negative curvature~\cite{inoueyano,
    connellwang}.
  \end{itemize}
  Moreover, the class of manifolds with positive simplicial volume
  is closed with respect to connected sums and products.
\end{example}

\begin{prop}[{\cite{vbc}\cite[Section~7.2]{Frigerio:book}}]\label{prop:degree:sv}
  Let $M$ and $N$ be oriented closed connected manifolds with 
  $M \geqone N$. Then
  $$
  \sv M \geq \sv N \ .
  $$ 
\end{prop}

\begin{rem}\label{rem:degree1maps:epi:mono}
  Proposition~\ref{prop:degree:sv} is an instance of the philosophy
  that if $M \geqone N$, then $M$ contains more topological information
  than~$N$.  We recall other classical results in this
  direction~\cite[Proposition~2.6]{Harpe:degree:one}: If $f \colon M
  \longrightarrow N$ is a continuous map of degree~$\pm 1$, then
  \begin{itemize}
  \item $\pi_1(f) \colon \pi_1(M) \longrightarrow \pi_1(N)$ is an
    epimorphism,
  \item $H_*(f;\Z) \colon H_*(M;\Z) \longrightarrow H_*(N;\Z)$
    is a split epimorphism,
  \item $H^*(f;\Z) \colon H^*(N;\Z) \longrightarrow H^*(M;\Z)$
    is a split monomorphism. 
  \end{itemize}
\end{rem}

\subsection{Amenable covers}

Let $\mathcal G$ be an isq-class of groups with~$\mathcal {G} \subset \Am$.
Then bounded cohomology and simplicial volume give obstructions to having
small $\mathcal{G}$-category.

\begin{thm}[Gromov's vanishing theorem~\cite{vbc}]\label{grom:van:thm}
  Let $\mathcal G$ be an isq-class of groups with~$\mathcal{G} \subset
  \Am$ and let $X$ be a topological space. Then: If $\gcat (X) \leq
  k$, then $\comp_X^s \colon H^s_b(X) \to H^s(X)$ vanishes for all $s
  \geq k$.
\end{thm}

\begin{rem}
  It is worth noticing that usually the previous theorem is stated in
  a different version, by requiring a control on the multiplicity of
  the cover instead of its cardinality~\cite{vbc}. More precisely, 
 if $X$ admits an open cover given by amenable 
 sets $\mathcal{U}$ with multiplicity at most $k$, 
 then $\comp_X^s \colon H^s_b(X) \to H^s(X)$ vanishes for all $s
  \geq k$.
  It is immediate to check that the bound on the cardinality of $\mathcal{U}$
   in Theorem~\ref{grom:van:thm} is a priori stronger than the one
  on the multiplicity of $\mathcal{U}$. On the other hand, by
  Lemma~\ref{lemma:reduce:cardinality:open:covers}, one can check 
  that these two
  formulations are equivalent if the space $X$ is a CW-complex.
\end{rem}

\begin{cor}\label{cor:catsimvol}
  If $M$ is an oriented closed connected manifold with $\sv M > 0$, 
  then $\amcat M = \dim M + 1$.
\end{cor}

\begin{proof}
  As mentioned in Remark~\ref{rem:duality:principle}, the positivity
  of the simplicial volume of~$M$ is equivalent to the surjectivity
  of~$\comp^n_M$. Hence, Theorem~\ref{grom:van:thm} implies the claim.
\end{proof}

Moreover, we introduce the family of spaces for which
the previous theorem is in fact an \emph{if and only if} statement:

\begin{defi}
  Let $\mathcal G$ be an isq-class of groups with~$\mathcal{G} \subset \Am$.
  A topological space~$X$ has \emph{$\efcat$} if
  \[ \gcat(X)
  = \max \bigl\{
  s \in \N \bigm|
  \comp_X^{s-1} \neq 0
  \bigr\}.
  \]
\end{defi}

\begin{rem}
  In this context it is interesting to consider the following question:
  Does every oriented closed connected \emph{aspherical} manifold
  have efficient $\Am$-category?

  If this were the case, then the vanishing of simplicial volume of
  aspherical closed connected manifolds would imply the vanishing of
  their $L^2$-Betti numbers~\cite{sauerminvol} and thus of their Euler
  characteristic. In particular, this would give an affirmative answer
  to the corresponding question by Gromov~\cite[p.~232]{gromovasym}.
\end{rem}

We conclude this section by discussing some difficulties in extending
classical lower bounds for the $\LS$-category to the amenable
setting:

\begin{rem}
	\label{rem:cup:length}
	A standard lower bound of the $\LS$-category is given by the
        cup length~{\cite[Proposition 1.5]{CLOT}}. A similar lower
        bound for the amenable category $\amcat(X)$ would give new
        insights into the cup product structure in bounded cohomology.
        Indeed, given a topological space $X$, we can define the
        \emph{bounded cup length} $\text{cup}_b(X)$ of $X$ as the
        maximal~$n\in\mathbb{N}_{\geq 0}$ such that there
        exist $\alpha_1,\dots, \alpha_n \in H^\bullet_b(X;\R)$ of
        positive degree with $\alpha_1\smile\dots \smile \alpha_n
        \neq0$.  A natural question is now to ask if the bounded cup
        length still provides a lower bound in the case of
        $\Am$-category:
	\[
	\textup{cup}_b(X) + 1 \leq \amcat(X)\;?
	\]
	A positive answer to this question would have relevant
        consequences to the computation of higher dimensional bounded
        cohomology groups.

        For instance, if $X$ is a wedge of at least two 
        circles, we have that its \emph{second} and \emph{third}
        bounded cohomology groups are infinite dimensional vector
        spaces, since the fundamental group of $X$ is a non-abelian
        free group~\cite{Mitsu, Soma3D}. On the other hand, it is an open 
        problem whether the bounded cohomology of~$X$ vanishes in
        higher degrees. A natural approach to investigate
        the problem is to study the non-vanishing of cup products
        between cohomology classes in degree $2$ or $3$. However,
        recently Heuer~\cite{Heuer:cup} and Bucher-Monod~\cite{BM:cup}
        proved independently that this approach fails at least for
        classes induced by Rolli and Brooks quasimorphisms. Hence, a
        positive answer in this situation to the question above would
        show that this is in fact the case of \emph{all} cup product
        in the bounded cohomology of $X$. Indeed, since $\amcat(X)=2$,
        we would get $\text{cup}_b(X)\leq 1$, whence the claim.
\end{rem}


\section{From complexes to manifolds}\label{sec:cplxs:mflds}

Using the standard embedding-thickening argument, examples of category
computations for finite complexes can be promoted to corresponding
examples of closed manifolds.

\begin{lemma}\label{lem:amcatmfd}
  Let $\mathcal G$ be an isq-class of groups, 
  let $k \in \N_{\geq 2}$, let $n \in \N_{\geq 2k}$, and let $X$ be a
  connected finite CW-complex of dimension at most~$k$. Then there
  exists an oriented closed connected $n$-manifold~$M$ with
  \[ \gcat M \leq \gcat X \qand \pi_1(M) \cong \pi_1(X).
  \]
\end{lemma}
\begin{proof}
  The CW-complex~$X$ is homotopy equivalent to a connected finite
  simplicial complex~$Y$ of dimension at most~$k$~\cite[Theorem~2C.5]{Hatcher}. Then
  $\pi_1(Y) \cong \pi_1(X)$ and $\gcat (Y) = \gcat (X)$
  (Proposition~\ref{prop:cathinv}). Therefore, it suffices to prove
  the claim for~$Y$.
  
  Since~$Y$ is a simplicial complex, we can embed it into~$\R^{n}$ for
  every $n \geq 2 \dim (Y) +1$~\cite[Theorem~1.6.1]{Mato:simp:emb}.
  Let $W \subset \R^{n}$ be a regular neighbourhood
  of~$Y$~\cite{RS-PL}.
  
  Then $W$ is an orientable compact connected
  $n$-manifold~\cite[Proposition~3.10]{RS-PL} with non-empty boundary
  and $W$ deformation retracts onto~$Y$~\cite[Corollary~3.30]{RS-PL}.
  Let us call this deformation
  retraction~$h$. In particular, $\pi_1(W) \cong \pi_1(Y)$.
  
  We define $$M \coloneqq \partial W$$ and we claim that $M$ is our
  desired oriented closed connected $n$-manifold. First notice that
  since~$k \geq 2$, the complex~$Y$ has codimension at least~$3$ in~$W$
  and~$\R^n$. This
  implies that the boundary~$M$ is connected and that the
  boundary inclusion $i_M \colon M \to W$ induces an
  isomorphism on fundamental groups.  Moreover, if we consider the
  composition
  $$
  h \circ i_M \colon M \to Y
  $$
  we obtain a map from $M$ to $Y$ inducing an isomorphism on the
  fundamental groups.  Hence, we have that $\pi_1(M) \cong \pi_1(Y)$
  and $\gcat (M) \leq \gcat (Y)$ (Remark~\ref{rem:pullbackcat}). This
  finishes the proof.
\end{proof}

Since in most of our applications we will make explicit use of a
retraction from $M$ onto $Y$, we prefer to state the following version 
in which we construct~$M$ as the double of the regular neighbourhood
$W$ of $Y$ in $\mathbb{R}^n$. This will produce a shift by $1$ on the
dimension of the resulting manifold.

\begin{prop}\label{prop:catY:catmfld}
  Let $\mathcal G$ be an isq-class of groups and let $Y$ be a simplicial
  complex of finite dimension.  Then, for every $n \geq 2 \dim(Y) +1$,
  there exists an oriented closed connected $n$-manifold $M$ such that
  \begin{enumerate}
  \item  $\pi_1(M) \cong \pi_1(Y)$;
  \item There exists a retraction $f \colon M \to Y$, i.e., $f \circ i_Y = \id_Y$, 
    where $i_Y$ is the inclusion of $Y$ into $M$.
  \end{enumerate}
  Moreover, if $Y$ is a model for~$B\Gamma$ of some finitely
  presented group~$\Gamma$ or if $\pi_1 (Y)$ is Hopfian, we have the
  following:
  \begin{enumerate}
  \item[(3)] $\gcat (M) \leq \gcat (Y)$;
  \item[(4)] If $Y$ has $\efcat$, then $M$ has $\efcat$ and $\gcat (M) = \gcat (Y)$.
  \end{enumerate}
  Finally, if $Y$ is an oriented closed connected manifold, then the
  previous result can be improved by taking $n \geq 2 \dim(Y)$.
\end{prop}

\begin{proof}
  \emph{Ad~1.} We already know that ~$Y$ can embedded in~$\R^{n}$ for
  every $n \geq 2 \dim (Y) +1$~\cite[Theorem~1.6.1]{Mato:simp:emb}.
  Let $W \subset \R^{n}$ be a regular neighbourhood of~$Y$ and let $h$
  be the deformation retraction from $W$ to $Y$.
  
  Since $Y$ has codimension at least~$2$ in~$W$, we know that the
  inclusion $i_{\partial W} \colon \partial W \to M$ of the boundary 
  induces an epimorphism on fundamental groups.
  
  We now claim that the double manifold $M \coloneqq D(W)$ is the
  desired oriented closed connected $n$-manifold.  First, notice that
  the fundamental group of $M$ is isomorphic to the amalgamated
  product $\pi_1(W) \ast_{\pi_1(\partial W)} \pi_1(W)$, where the
  morphisms $\pi_1(\partial W) \to \pi_1(W)$ are given by
  $\pi_1(i_{\partial W})$.  Since the latter map is an epimorphism, we
  have that $$\pi_1(M) \cong \pi_1(W) \ast_{\pi_1(\partial W)}
  \pi_1(W) \cong \pi_1(W) \cong \pi_1(Y) \ .$$
   
  \emph{Ad~2.} By construction, there exists a retraction
  $r \colon M \to W$, i.e., $r \circ i_W = \id_W$ where $i_W$
  denotes the inclusion of~$W$ into its double~$M$.  Then 
  the composition
  \begin{equation}\label{eq:composition:classifying}
  h \circ r \colon M \to Y 
  \end{equation}
  is the claimed retraction.
  
  \emph{Ad~3.} By Remark~\ref{rem:pullbackcat}, it is sufficient to
  construct a map~$f \colon M \to Y$ that induces an isomorphism on
  fundamental groups.
  
  Now, we have two cases. Let us suppose that $\pi_1(Y)$ is
  Hopfian. Then the retraction in
  Equation~(\ref{eq:composition:classifying}) induces an epimorphism 
  $\pi_1(M) \cong \pi_1(Y) \to \pi_1(Y)$, whence an
  isomorphism. This shows that $\gcat (M)
  \leq \gcat (Y)$.
  
  On the other hand, if we assume that $Y$ is a model for $B\Gamma$,
  there exists a classifying map $f \colon M \to Y$ inducing an
  isomorphism on fundamental groups.  Hence, also in this case we get
  $\gcat (M) \leq \gcat (Y)$. This concludes the proof of the claim.
  
  \emph{Ad~4.} We already know from the second part that $\gcat (M) \leq
  \gcat (Y)$. Hence, we only have to prove the converse.
  Since $Y$ has $\efcat$, we can assume that $\gcat (Y) =
  k + 1$ and $\comp_Y^k$ is non-trivial. Let us consider the map~$f
  \coloneqq h \circ r \colon M \to Y$ defined in the proof of the
  second part.  Then we have the following commutative diagram:
	$$
	\xymatrix{
	H^k_b(Y) \ar[rr]^-{H^k_b(f)} \ar[d]_-{\comp^k_Y} && H^k_b(M) \ar[d]^-{\comp^k_M} \\
    H^k(Y) \ar[rr]_-{H^k(f)} && H^k(M) \ .
    }
   $$     
   Since $f$ is a retraction, we have that the two
   horizontal arrows are monomorphisms.  Hence, the non-vanishing
   of~$\comp^k_Y$ also implies the
   non-vanishing of~$\comp^k_M$. By applying
   Theorem~\ref{grom:van:thm}, we get $\gcat (M) \geq k = \gcat (Y)$.
   This finishes the proof of this item.
  
   The last statement comes from the fact that if $Y$ is an orientable
   closed connected manifold, then the strong Whitney embedding
   theorem shows that $Y$ can be embedded in $\R^{n}$ for every $n
   \geq 2 \dim(Y)$.
\end{proof}

\begin{rem}\label{rem:improveddim}
  When using Proposition~\ref{prop:catY:catmfld},
  we will always be able to obtain a corresponding
  version with the improved dimension bound, provided
  that we use a manifold model of the classifying space
  as input. Usually, we will not formulate these
  improved versions explicitly.
\end{rem}

\begin{rem}\label{rem:resulting:manifold:triangulable}
Notice that by construction the resulting manifold~$M$ 
in Proposition~\ref{prop:catY:catmfld} is oriented closed connected
and \emph{triangulable}.
\end{rem}

An oriented closed connected triangulable $n$-manifold $M$ 
is \emph{essential} in the sense of Gromov~\cite[4.40]{gromovmetric}
if for every map $$f \colon M \to B\pi_1(M)$$ inducing an isomorphism on fundamental
groups, the image $f(M)$ is not contained in the $(n-1$)-skeleton of $B\pi_1(M)$.
In particular, this is the case of an oriented closed connected triangulable $n$-manifold 
$M$ which admits a 
classifying map $f \colon M \to B\pi_1(M)$ inducing a non-trivial map
$$
H^n(f) \colon H^n(B\pi_1(M)) \to  H^n(M)
$$
on the \emph{real} $n$-th cohomology groups.

As introduced by Dranishnikov, Katz, and Rudyak~\cite[Definition~5.1]{DKR} one can extend Gromov's definition 
as follows: Let $k > 1$. Then, an oriented closed connected triangulable
$n$-manifold $M$ is said to be \emph{strictly} $k$-\emph{essential} if 
there is no map between skeleta
$$
f \colon M^{(k)} \to B\pi_1(M)^{(k-1)}
$$
inducing an isomorphism on fundamental groups. 
Clearly, an oriented closed connected triangulable $n$-manifold is 
essential if and only if it is strictly $n$-essential. 
We then introduce the following definition:

\begin{defi}
  Let $k > 1$. An oriented closed connected triangulable $n$-manifold $M$ is said to be \emph{essential
  in degree} $k$ if the classifying map $f \colon M \to
  B\pi_1(M)$ induces a non-trivial map on the real $k$-th cohomology
  groups
  $$
  H^k(f) \colon H^k(B\pi_1(M)) \to  H^k(M) \ .
  $$
\end{defi} 

\begin{rem}\label{rem:essential:in:degree:vs:strictly}
  It is immediate to check that if an oriented closed connected triangulable
  $n$-manifold is essential in degree $k > 1$, then it is also strictly $k$-essential.
  On the other hand, the converse is false.
\end{rem}

\begin{cor}\label{cor:hopfian:strictly:essential}
  Let $\mathcal G$ be an isq-class of groups and
  let $Y$ be an oriented closed connected triangulable $n$-manifold
  that is essential in degree $k > 1$. Moreover, suppose that $\pi_1(Y)$
  is Hopfian.

  Then, for every $n \geq 2 \cdot
  \dim (Y)$, there exists an oriented closed connected triangulable $n$-manifold~$M$ essential 
  in degree $k$ such that
  $$
  \gcat (M) \leq \gcat (Y)  \qand \pi_1(M) \cong \pi_1(Y).
  $$
\end{cor}
\begin{proof}
  Let $g \colon Y \to B\pi_1(Y)$ be
  the classifying map of $Y$. Since $Y$ is essential in degree~$k$ we have that 
  $$
  H^k(g) \colon H^k(B\pi_1(Y)) \to H^k(Y)
  $$
  is non-trivial.
  
  Let us now construct an oriented closed connected triangulable
   $n$-manifold $M$ from~$Y$ as explained in  
  Proposition~\ref{prop:catY:catmfld} and Remark~\ref{rem:resulting:manifold:triangulable}.  
  Using Proposition~\ref{prop:catY:catmfld}.2, we obtain a retraction
  map~$f \colon M \to Y$;
  in particular, $f$ induces an injection in both bounded and
  ordinary cohomology. This shows that the composition
  $$
  H^k(g \circ f) \colon H^k(B\pi_1(Y)) \to H^k(M)
  $$
  is non-trivial. Hence, we can conclude that $M$ is essential in degree $k$
  if we show that $g \circ f$ is in a fact a classifying map for $M$. To this end, 
  notice that $f$ is a retraction and $\pi_1(Y)$ is Hopfian, thus
  $f$ is  also a $\pi_1$-isomorphism.  This shows that the composition 
  $g \circ f$ is a classifying map for $M$, whence that $M$ is essential in
  degree $k$. Then the thesis follows by applying
  Proposition~\ref{prop:catY:catmfld}.3.
\end{proof}

\section{Small amenable category and the fundamental group}\label{sec:smallamcat}

We will now prove Theorem~\ref{thm:smallamcat}. We begin with the following statement 
that corresponds to the first item of Theorem~\ref{thm:smallamcat}.


\begin{prop}\label{prop:amcat3}
  Let $\Gamma$ be a finitely presented group and let $n \in \N_{\geq
    4}$. Then there exists an oriented closed connected
  $n$-manifold~$M$ with~$\pi_1(M) \cong \Gamma$ and $\amcat M \leq 3$.
\end{prop}
\begin{proof}
  Let $X$ be the presentation complex of a finite presentation
  of~$\Gamma$.  Then $X$ is a connected finite CW-complex of
  dimension~$2$ with~$\pi_1(X) \cong \Gamma$. Applying
  Lemma~\ref{lem:amcatmfd} to~$X$ gives a manifold with the claimed
  properties.
\end{proof}


It is instructive to keep the following example in mind:

\begin{example}
  Let $n \in \N_{\geq 3}$ and let $T^n$ denote the $n$-torus. Then
  \[ \amcat (T^n \connsum T^n) = 2
     \qand \pi_1(T^n \connsum T^n) \cong \Z^n * \Z^n.
  \]
  This can be seen as follows: The fundamental group is not
  amenable; thus, $\amcat$ has to be at least~$2$. Conversely,
  because $n \geq 3$, the two punctured torus summands give
  rise to an amenable cover.
\end{example}

\begin{prop}\label{prop:pi1amcat2}
  Let $X$ be a connected CW-complex, let $F$ be a subgroup family
  of~$\pi_1(X)$ and let~$\catop_F X \leq 2$. Then the fundamental
  group of~$X$ is the fundamental group of a graph of groups whose
  vertex (and edge) groups are all in~$F$.

  In particular: If $\amcat X \leq 2$, then the fundamental group of~$X$
  is the fundamental group of a graph of groups whose vertex (and edge)
  groups are all amenable.
\end{prop}
\begin{proof}
  We argue via Bass-Serre theory: 
  Because $\catop_F X \leq 2$, there exists an $F$-cover~$U$ of~$X$
  by path-connected subsets such that the nerve~$N(U)$ is
  one-dimensional: We decompose an open $F$-cover of~$X$ by two
  sets into their path-connected components; these sets are open as
  CW-complexes are locally path-connected.

  Let $\widetilde U$ be the corresponding open cover of the universal
  covering~$\widetilde X$ of~$X$ (where $\pi \colon \widetilde X
  \longrightarrow X$ denotes the universal covering map):
  \begin{align*}
    \widetilde U
    := \bigl\{ V \subset \widetilde X
    \bigm|
    & \;\text{there exists a~$W \in U$ such that}
    \\[-.5ex]
    & \;\text{$V$ is a path-connected component of~$\pi^{-1}(W)$}
    \bigr\}
  \end{align*}
  Then the nerve~$N(\widetilde U)$ is also one-dimensional. Moreover,
  $N(\widetilde U)$ is a forest: Because $H_1(\widetilde X;\Z) \cong
  0$, a Leray spectral sequence argument shows that
  $H_1(N(\widetilde U);\Z) \cong 0$~\cite[Theorem~2.1]{meshulam};
  hence, all connected components of~$N(\widetilde U)$ are trees.
  
  The fundamental group~$\Gamma$ of~$X$ acts simplicially
  on~$N(\widetilde U)$ and the stabiliser groups of the vertices
  and edges all lie in~$F$~\cite[Lemma~4.11]{LS}.

  Let $T$ be the barycentric subdivision of a connected component
  of~$N(\widetilde U)$. Then $T$ is a tree that inherits an
  involution-free simplicial action of~$\Gamma$ whose vertex and edge
  stabilisers lie in~$F$. Thus, by Bass-Serre
  theory~\cite{serretrees}, the group~$\Gamma$ is isomorphic to the
  fundamental group of the graph of groups on~$T/\Gamma$, decorated by
  these stabiliser groups.

  Alternatively, one can apply the generalised Seifert and van~Kampen
  theorem for fundamental groupoids~\cite{brownsalleh,drorfarjoun}
  and relate the resulting homotopy colimit of groupoids to fundamental
  groups of graphs of groups via explicit presentations~\cite{higgins}.
\end{proof}

In the manifold case, open amenable covers consisting of exactly two
open amenable sets can also be arranged by appropriate closed
submanifolds with common boundary~\cite[Lemma~3]{GGW}.

\begin{cor}\label{cor:amcat2}
  Let $\Gamma$ be a non-amenable group. Then the following are equivalent:
  \begin{enumerate}
  \item The group~$\Gamma$ is the fundamental group of a graph of
    groups whose vertex (and edge) groups are all amenable.
  \item If $X$ is a model of~$B\Gamma$, then $\amcat X = 2$.
  \item If $X$ is a CW-complex with~$\pi_1(X) \cong \Gamma$, then $\amcat X = 2$.
  \item There exists a connected CW-complex~$X$ that satisfies~$\pi_1(X) \cong
    \Gamma$ and $\amcat X = 2$.
  \end{enumerate}
  If $\Gamma$ is finitely presented, then these conditions are also
  equivalent to:
  \begin{enumerate}
    \setcounter{enumi}{4}
  \item There exists an oriented closed connected manifold~$M$
    that satisfies~$\pi_1(M) \cong \Gamma$ and $\amcat M = 2$.
  \end{enumerate}
\end{cor}
\begin{proof}
  We begin with~$(1) \Longrightarrow (2)$. Let $\Gamma$ be the
  fundamental group of a graph of groups whose vertex (and edge)
  groups are all amenable. Then the homotopy colimit~$X$ over a
  corresponding graph of classifying spaces is a model of the
  classifying space~$B\Gamma$~\cite[Section~4.1]{drorfarjoun}. Moreover, the underlying
  graph structure gives an amenable open cover of~$X$ consisting of
  two sets (namely, one associated to the vertices and one associated
  with the edges).  Therefore, $\amcat X \leq 2$. Because $\Gamma$ is
  non-amenable, we obtain~$\amcat X = 2$.  As $\amcat$ is a homotopy
  invariant (Proposition~\ref{prop:cathinv}), also all other models
  of~$B\Gamma$ have amenable category equal to~$2$.

  The implication~$(2) \Longrightarrow (3)$ follows by pulling back
  amenable open covers of models of~$B\Gamma$ along the classifying
  map (and the fact that $\Gamma$ is non-amenable).

  For~$(3) \Longrightarrow (4)$, we only need to notice that
  for every group~$\Gamma$ there exists a CW-complex whose fundamental
  group is isomorphic to~$\Gamma$. 
  
  The implication~$(4) \Longrightarrow (1)$ is covered by
  Proposition~\ref{prop:pi1amcat2}.
  
  For the implication~$(4) \Longrightarrow (5)$, we apply
  Lemma~\ref{lem:amcatmfd}. The implication~$(5) \Longrightarrow (4)$
  is a consequence of the fact that every compact manifold is homotopy
  equivalent to a finite CW-complex~\cite{kirbysiebenmann,siebenmann}
  and that $\amcat$ is homotopy invariant
  (Proposition~\ref{prop:cathinv}).
\end{proof}

This proposition also holds in the case of general isq-classes
of groups (instead of~$\Am$), under the assumption that $\Gamma$ is
not contained in this class.

\begin{cor}\label{cor:amcatnot2}
  Let $M$ be an oriented closed connected manifold whose fundamental
  group is in the following list:
  \begin{enumerate}
  \item Non-amenable groups with Serre's property~FA.
  \item Fundamental groups of oriented closed connected aspherical
    manifolds of dimension at least~$2$ with positive simplicial volume.
  \item
    Groups that have a non-zero $L^2$-Betti number in some degree~$\geq 2$.
  \end{enumerate}
  Then~$\amcat M \geq 3$.
\end{cor}
\begin{proof}
  Let $\Gamma := \pi_1(M)$ and let $X$ be a model of~$B\Gamma$. In view
  of Corollary~\ref{cor:amcat2}, we only need to show that $\amcat X \geq 3$.
  \begin{enumerate}
  \item This is immediate from the definition of property~FA and
    Bass-Serre theory~\cite{serretrees}.
  \item In the simplicial volume case, this is a consequence of the vanishing
    theorem in bounded cohomology
    (Theorem~\ref{grom:van:thm}).
  \item In the $L^2$-Betti number case, this is a consequence of Sauer's
    vanishing theorem~\cite[Theorem~C]{sauerminvol}\cite{gromovmetric}.
    \qedhere
  \end{enumerate}
\end{proof}

Corollary~\ref{cor:amcatnot2} can, for instance, be applied to
infinite groups with property~(T)~\cite{watatani}, to the
examples listed in Example~\ref{exa:possimvol}, and
to~$F_2 \times F_2$ (which has non-zero second $L^2$-Betti number).

Using the Berstein class, Dranishnikov and Rudyak showed the
following~\cite[Theorem~5.2]{DR}: If $M \geqone N$ and
$\pi_1(M)$ is free, then also $\pi_1(N)$ is free. 
In the setting of amenable category, this translates into the
following question:

\begin{quest}\label{q:DRam}
  Let $M$, $N$ be oriented closed connected manifolds with~$M \geqone N$
  and let $\pi_1(M)$ be the fundamental group of a graph of groups
  all of whose vertex (and edge) groups are amenable. Is then also
  $\pi_1(N)$ the fundamental group of a graph of groups all of
  whose vertex (and edge) groups are amenable?
\end{quest}

The strategy of proof of Dranishnikov and Rudyak is based on the
characterisation of free groups in terms of cohomological dimension
and thus does not seem to have a straightforward counterpart in the
case of bounded cohomology and amenable category. However, a positive
answer to the generalised monotonicity problem
(Question~\ref{q:ammono}) for domain manifolds of amenable category
equal to~$2$ would combine with Corollary~\ref{cor:amcat2} to give an
affirmative answer to Question~\ref{q:DRam}.

\section{Degree-one maps and categorical invariants}\label{sec:degree:one}

In this section, we focus on Question~\ref{q:ammono}, i.e., on the
generalisation of Rudyak's problem (Question~\ref{quest:rudyak}). 
By now there are several positive results on Rudyak's
question~\cite{Rudyak:question, Rudyak:maps:new, DS:connected1,
  DS:rudyak:20}. For instance, Rudyak showed that the question has
a positive answer for manifolds of dimension at most~$4$~\cite{Rudyak:maps:new}.
Moreover, it is also true if $N$ has
$\LS$-category at most~$3$. Indeed, if $\lscat (M) = 2$,
then $M$ is a homotopy sphere~\cite{James}. This then implies that $N$
is also a homotopy sphere and thus has LS-category~$2$~\cite{Rudyak:maps:new}.

Moreover, a good source of degree-one maps is provided by collapsing maps from
the connected sum onto one of its summands. It has been proved by Dranishnikov
and Sadykov~\cite{DS:connected1, DS:connected2} that Rudyak's question
in this situation always has a positive answer. We refer the reader
to the literature~\cite{Rudyak:maps:new, DS:rudyak:20} for examples of higher
connected manifolds that satisfy Rudyak's question.

Finally, it is worth mentioning that Rudyak's question
is answered affirmatively when the $\LS$-category of~$N$ is
maximal~\cite{Rudyak:question}. This is our source of inspiration for
studying the generalised Rudyak's question for classes of groups
in the next section. Indeed, by using
simplicial volume it is easy to prove that if $N$ has positive
simplicial volume, then $N$ satisfies Rudyak's question:

\begin{lemma}\label{lemma:volsim:pos:rudyak}
  If $M$ and $N$ are 
  oriented closed connected manifolds with $M \geqone N$ and 
  $\sv N > 0$, then $\lscat (M) \geq \lscat (N)$.
\end{lemma}
\begin{proof}
  By Corollary~\ref{cor:catsimvol}, we have that $\lscat (N) = n +1$.
  Hence, by using Proposition~\ref{prop:degree:sv} and again
  Corollary~\ref{cor:catsimvol}, it is readily seen that also $\lscat
  (M) = n+1$, whence the thesis.
\end{proof}

\begin{rem}\label{rem:simvol:works:for:G:not:only:LS}
Notice that the proof of Lemma~\ref{lemma:volsim:pos:rudyak}
 also applies to the following situation: Let $\mathcal{G}$ be an isq-class of
 groups with $\mathcal{G} \subset \Am$. Then, for all oriented closed connected 
 manifolds $M$ and $N$ with $M \geqone N$ and $\sv N > 0$, we have
 $\gcat (M) \geq \gcat (N)$. This observation leads to the study of the 
 generalised monotonicity problem in the following section.
\end{rem}

\subsection{The monotonicity problem}

We will now extend Rudyak's question~\ref{quest:rudyak} to the
setting of categories with respect to isq-classes of amenable
groups:

\begin{quest}[(generalised) monotonicity problem]\label{q:ammono:new}
  Let $\mathcal G$ be an isq-class of groups with~$\mathcal{G} \subset \Am$. 
  Does the following hold for
  all oriented closed connected manifolds~$M$ and~$N$:
  \[ M \geqone N \Longrightarrow \gcat M \geq \gcat N \text{\;?}
  \]
\end{quest}

\begin{rem}\label{rem:gen:rud:quest:true:leq:2}
  It is immediate to check that the previous question is always true
  if $\gcat(N) \leq 2$. Indeed, if $\gcat(N) = 1$, there is nothing
  to prove. On the other hand, if
  $\gcat(M) = 1$, then $\pi_1(M) \in
  \, \mathcal{G}$. Since $\mathcal{G}$ is closed under quotients, we
  have that also $\pi_1(N) \in \, \mathcal{G}$
  (Remark~\ref{rem:degree1maps:epi:mono}). This shows that $\gcat (N)
  = \gcat(M) = 1$.
\end{rem}


\subsection{Low-dimensional manifolds}

We answer the monotonicity problem (Question~\ref{q:ammono})
affirmatively in dimensions~$2$ and~$3$.

\begin{prop}\label{prop:gen:R:quest:surfaces}
  Let $\mathcal G$ be an isq-class of groups with~$\mathcal{G} \subset \Am$.
  Then all oriented
  closed connected surfaces $M$ and $N$ with $M \geqone N$
  satisfy $\gcat (M) \geq \gcat(N)$.
\end{prop} 
\begin{proof}
  We use the classification of surfaces: For~$g \in \N$, we
  write~$\Sigma_g$ for ``the'' oriented closed connected surface of
  genus~$g$. Looking at~$H_1(\args;\Z)$ shows that (Remark~\ref{rem:degree1maps:epi:mono})
  \[ \fa{g,h \in \N} \Sigma_g \geqone \Sigma_h \Longrightarrow g \geq h.
  \]
  Therefore, the claim is immediate from the following computations:
  \begin{itemize}
  \item $\gcat(\Sigma_g) = 3$ for all~$g\in \N_{\geq 2}$ because
    $\sv{\Sigma_g} > 0$ (Example~\ref{exa:possimvol},
    Corollary~\ref{cor:catsimvol}) and $\mathcal{G} \subset \Am$.
  \item $\gcat(\Sigma_1) \in \{1,2,3\}$ by the estimates from Remark~\ref{rem:F-cat:leq:lscat}.
  \item $\gcat (\Sigma_0) = 1$, because $\Sigma_0$ is simply connected
    and thus a $\mathcal{G}$-set.
    \qedhere
  \end{itemize}
\end{proof}

\begin{thm}\label{thm:gen:R:quest:3:mflds}
  Let $\mathcal G$ is an isq-class of groups with~$\mathcal{G} \subset \Am$
  that contains the class of solvable groups.
  Then, for all oriented
  closed connected $3$-manifolds $M$ and~$N$ with $M \geqone N$, we have $\gcat
  (M) \geq \gcat (N)$.
\end{thm}

Before proving this theorem we need some preparation. First, we mention
the behaviour of the $\mathcal{G}$-category under connected sums (see
also Corollary~\ref{cor:conn:sum:gen:R:quest} below).

\begin{prop}[{\cite[Lemma~1]{GGW}}]\label{prop:connected:sum}
  Let $\mathcal G$ be an isq-class of groups with $\mathcal G \subset \Am$ and 
  let $n \geq 3$. Let
  $M_1$ and $M_2$ be oriented closed connected $n$-manifolds with 
   $\max \{ \gcat (M_1), \gcat (M_2)\} \geq 2$. Then, we have
  $$
  \gcat (M_1 \connsum M_2) = \max \{ \gcat (M_1), \gcat (M_2)\} \ .
  $$
\end{prop}


\begin{proof}
  Notice that the proof~\cite[Lemma~1]{GGW} for $3$-manifolds such
  that both $\gcat (M_1)$ and $\gcat (M_2) \geq 2$ applies
  \emph{verbatim} for all manifolds of dimension greater than or equal
  to $3$ such that $\gcat (M_1), \gcat (M_2) \geq 2$. Moreover, the
  very same proof~\cite[Lemma~1]{GGW} can be also extended to the case
  in which either $\gcat (M_1)$ or $\gcat (M_2)$ is equal to~$1$.
\end{proof}

\begin{rem}\label{rem:connected:sum}
  If $\max \{ \gcat (M_1), \gcat (M_2)\} = 1$, then the value of
  $\gcat (M_1 \connsum M_2)$ depends both on $\mathcal{G}$ and the
  manifolds~$M_1$, $M_2$. For instance if $\mathcal{G} \subset \Am$ is
  an isq-class of groups, then $\gcat (M_1 \connsum M_2) = 2$ if and only
  if $\pi_1(M_1) \ast \pi_1(M_2)$ is a non-trivial free product of
  amenable groups.
\end{rem}

As a corollary, we get some instances for the validity of the
monotonicity problem (Question~\ref{q:ammono}), similar to
corresponding results for LS-category~\cite{DS:connected1}.  Recall
that given two oriented closed connected $M_1$ and $M_2$ of the same
dimension, there always exists a degree-one map~$M_1 \connsum M_2 \to
M_i$ for $i \in \, \{1, 2\}$.  Hence,
we have the following:

\begin{cor}\label{cor:conn:sum:gen:R:quest}
  Let $\mathcal G$ be an isq-class of groups and 
  let $M_1$ and $M_2$ be 
  oriented closed connected $n$-manifolds with $n \geq 2$. Then $M_1
  \connsum M_2 \geqone M_i$ and $\gcat (M_1 \connsum M_2) \geq \gcat
  (M_i)$ for both~$i \in \{1, 2\}$.
\end{cor}

\begin{proof}
  The case for $n = 2$ has been already discussed in
  Proposition~\ref{prop:gen:R:quest:surfaces}. If $n \geq 3$, by
  applying Proposition~\ref{prop:connected:sum} we have that $\gcat
  (M_1 \connsum M_2) \geq \gcat (M_i)$ for both~$i \in \{1,2\}$ if at least
  one of the manifolds~$M_1$ and $M_2$ has $\mathcal{G}$-category
  strictly larger than~$1$. However, if we have assume that $\gcat
  (M_1) = 1 = \gcat (M_2)$, then the thesis follows trivially. 
\end{proof}

As a second ingredient for the proof of Theorem~\ref{thm:gen:R:quest:3:mflds} we 
recall the notion of Kodaira dimension of $3$-manifolds introduced by
Zhang~\cite{Zhang}. We follow here the notation by
Neofytidis~\cite{Neofytidis}, which is slightly different from Zhang's
original one. To this end we first
divide Thurston's three-dimensional eight geometries into four classes
assigning to each of them a value:
\begin{itemize}
\item $- \infty \colon S^3, S^2 \times \mathbb{R}$;
\item $0 \colon \mathbb{E}^3, \mathrm{Nil}^3, \mathrm{Sol}^3$;
\item $1 \colon \mathbb{H}^2 \times \mathbb{R}, \widetilde{\textup{SL}_2(\mathbb{R})}$;
\item $\frac{3}{2}: \mathbb{H}^3$.
\end{itemize}
Given an oriented closed connected $3$-manifold $M$ we decompose it
via Milnor-Kneser's prime decomposition and then we subdivide each
piece in the prime decomposition via a geometric decomposition
(i.e., we cut along tori such that each final manifold carries one of
the eight geometries and has finite volume). We call the resulting
decomposition a $T$-\emph{decomposition} of $M$.

\begin{defi}
  The \emph{Kodaira dimension}~$\kappa^t(M)$ of an oriented closed
  connected $3$-manifold $M$ is defined as follows:
  \begin{itemize}
  \item We set $\kappa^t(M): = -\infty$ if for any $T$-decomposition of $M$
    all its pieces belong to the category $-\infty$ above;
  \item We set $\kappa^t(M) := 0$ if any $T$-decomposition of $M$ contains
    at least one piece lying in category $0$, but no pieces in the
    categories $1$ or~$\frac{3}{2}$;
  \item We set $\kappa^t(M) := 1$ if any $T$-decomposition of $M$ contains
    at least one piece lying in category $1$, but no pieces in the
    category $\frac{3}{2}$;
  \item We set $\kappa^t(M) := \frac{3}{2}$ if any $T$-decomposition of
    $M$ contains at least one piece lying in category $\frac{3}{2}$.
\end{itemize}
\end{defi}

\begin{rem}\label{rem:geom:pieces}
  Notice that if a $3$-manifold $M$ has $\kappa^t (M) \leq 0$, then
  all the irreducible pieces in its prime decomposition are closed
  geometric manifolds~\cite{Zhang}.
\end{rem}

\begin{thm}[{\cite[Theorem~6.1]{Neofytidis} and~\cite[Theorem~1.1]{Zhang}}]\label{thm:kodaira}
  Let $M$ and $N$ be oriented closed connected $3$-manifolds with~$M
  \geqone N$.  Then $\kappa^t (M) \geq \kappa^t (N)$.
\end{thm}

Finally, we will also use previous computations of categorical
invariants of $3$-manifolds:

\begin{rem}[{\cite[Corollary~5]{GGW}}]\label{rem:solv3}
   Let $\mathcal G$ be an isq-class of groups with~$\mathcal{G} \subset \Am$
  that contains the class of solvable groups.
  G\'omez-Larra\~naga, Gonz\'alez-Acu\~na, and Heil computed
  the $\mathcal{G}$-category for all prime $3$-manifolds. 
  More precisely, if $M$ is prime manifold
  with $\kappa^t (M) \leq 0$, then the fundamental group is solvable, whence
  $\gcat(M) = 1$. If $\kappa^t (M) = 1$, then the fundamental group of
  $M$ is not solvable and $\gcat (M) = 3$. Finally, if $\kappa^t(M) = \frac{3}{2}$,
  we know that that simplicial volume of $M$ is positive~\cite{Soma81} and so
  we have $\gcat(M) = 4$ (Corollary~\ref{cor:catsimvol}).
\end{rem}

We are now ready to prove Theorem~\ref{thm:gen:R:quest:3:mflds}:

\begin{proof}[Proof of Theorem~\ref{thm:gen:R:quest:3:mflds}]
  Suppose that $M \geqone N$. We begin by considering the case in
  which $M$ is prime. Notice that there are no oriented closed
  connected prime $3$-manifolds with $\gcat$ equal to
  $2$~\cite{GGH:am}\cite[Corollary~3]{GGW} (Remark~\ref{rem:solv3}).
  So in this situation we have only three cases:
  \begin{enumerate}
  \item If $\pi_1(M)$ is solvable, then we have that $\gcat (M) =
    1$ because $\mathcal G$ contains the class of all solvable groups.
    This implies that also $\pi_1(N)$ is solvable
    (Remark~\ref{rem:degree1maps:epi:mono}), whence $\gcat (N)
    =1$. In particular, $\gcat (M) \geq \gcat (N)$.
  \item If $\gcat (M) = 3$, then Theorem~\ref{grom:van:thm} shows that
    $\sv M = 0$ and so $\sv N = 0$
    (Proposition~\ref{prop:degree:sv}). Hence, we have $\gcat (N) \leq
    3$~\cite[Theorem~2]{GGH:am}\cite[Corollary~5]{GGW} (Remark~\ref{rem:solv3}).
    In particular, $\gcat (M) \geq \gcat (N)$.
  \item If $\gcat (M) = 4$, then the desired result $\gcat (M) \geq
    \gcat (N)$ is trivially true by the dimension estimate (Remark~\ref{rem:F-cat:leq:lscat}).
  \end{enumerate}
  We now assume that $M$ is not prime and we argue via Kodaira
  dimension~$\kappa^t$. We have the following four cases:
  \begin{enumerate}
  \item Let $\kappa^t(M) = -\infty$. By Remark~\ref{rem:geom:pieces}
    and Remark~\ref{rem:solv3}, we know that $M$ is the connected sum
    of oriented prime geometric manifolds with $\gcat$ equal to
    $1$. Hence, we have $\gcat (M) \leq 2$ by
    Proposition~\ref{prop:connected:sum} and
    Remark~\ref{rem:connected:sum}. If $\gcat (M) =1$ we are done as
    mentioned in the first item above. Let $\gcat (M) = 2$. By
    Theorem~\ref{thm:kodaira}, we have $\kappa^t(M) \geq
    \kappa^t(N)$. This implies that also $\kappa^t(N) =
    -\infty$. Hence, we have $\gcat (N) \leq 2 = \gcat (M)$.
  \item Let $\kappa^t(M) = 0$. By Remark~\ref{rem:geom:pieces}
    and Remark~\ref{rem:solv3}, we have again that $\gcat (M) \leq
    2$. Since we still have $\kappa^t(M) \geq \kappa^t(N)$, the very
    same proof as in the previous item applies to this situation.
  \item Let $\kappa^t(M) = 1$. Let $M = M_1 \connsum \cdots \connsum
    M_k$ be a prime decomposition. Up to reordering the pieces, we may
    assume that $M_1$ contains a piece belonging to category
    $1$. Then by applying Remark~\ref{rem:solv3}, we have $\gcat
    (M_1) = 3$~\cite[Theorem~2]{GGH:am}. Hence, Proposition~\ref{prop:connected:sum} shows that
    $\gcat (M) = 3$. Then, Theorem~\ref{grom:van:thm} and
    Proposition~\ref{prop:degree:sv} imply that $0 = \sv M \geq \sv
    N$. The vanishing of the simplicial volume of $N$ then 
    implies that $\gcat (N) \leq 3$ (Remark~\ref{rem:solv3}), whence
    $\gcat (M) \geq \gcat (N)$.
  \item Let $\kappa^t(M) = \frac{3}{2}$. In this situation, we have
      $\gcat (M) = 4$ since $\sv M > 0$~\cite{Soma81}. Hence, the inequality $\gcat
    (M) \geq \gcat (N)$ follows from the dimension estimate
    (Remark~\ref{rem:F-cat:leq:lscat}). 
    \qedhere
\end{enumerate}
\end{proof}

\begin{rem}\label{rem:3mfldsdirect}
  There is also a way to formulate the proof of
  Theorem~\ref{thm:gen:R:quest:3:mflds} without using the Kodaira
  dimension. We are grateful to Dieter Kotschick for sharing such a
  proof with us. The rough outline is as follows: Using the
  calculations of G\'omez-Larra\~naga, Gonz\'alez-Acu\~na, and
  Heil~\cite{GGH:am}  and the computation of simplicial volume of
  $3$-manifolds, one can derive that we have for all oriented
  closed connected $3$-manifolds~$M$: 
  \begin{itemize}
  \item $\amcat M = 1$ if and only if $\pi_1(M)$ is amenable.
  \item If $\amcat M  =2$, then $M$ has at least two prime summands
    and all prime summands of~$M$ have amenable fundamental group. The converse
    also holds except for the pathological case~$\R P^3 \connsum \R P^3$
    (which has amenable category~$1$).
  \item $\amcat M = 4$ if and only if~$\|M\| > 0$.
  \item $\amcat M = 3$ if and only if all prime summands of~$M$ have
    vanishing simplicial volume and at least one prime summand has
    non-amenable fundamental group.
  \end{itemize}
  If there exists a map~$M \longrightarrow N$ between oriented closed
  connected manifolds of non-zero degree, one can then proceed by
  a case-by-case analysis for the different values of amenable category.
  The most interesting case is to exclude the option~$\amcat M =2$
  when $\amcat N  = 3$. For this case, one uses a result of
  Wang~\cite[Lemma~3.4]{wang} and basic inheritance properties of amenable
  groups. 
\end{rem}


\subsection{Using bounded cohomology}

We will use bounded cohomology and simplicial volume to give
sufficient conditions for a positive answer to the monotonicity
problem (Question~\ref{q:ammono}).  This approach will provide
infinite families of target manifolds satisfying the monotonicity
problem. 

\begin{rem}
  It is worth mentioning that working with bounded cohomology and
  simplicial volume, one can also introduce the notion of
  $\ell^1$-\emph{invisible} manifolds~\cite{Loeh:ell1}. Recently it
  has been proved~\cite{FM:Grom, LS, Frigerio:ell1} that a sufficient
  condition for an oriented closed connected $n$-manifold to be
  $\ell^1$-invisible is having~$\amcat \leq n$. Hence, an argument
  similar to the one in Lemma~\ref{lemma:volsim:pos:rudyak} 
  (Remark~\ref{rem:simvol:works:for:G:not:only:LS}) shows that
  manifolds that are not $\ell^1$-invisible satisfy the monotonicity
  problem.

  It is known that $\ell^1$-invisible manifolds have zero simplicial
  volume, but the converse implication is still unknown.  For
  manifolds of dimension at most~$3$, the vanishing of simplicial
  volume and $\ell^1$-invisibility are equivalent; this can be seen
  via the computation of amenable category~\cite{GGH:am} or via
  classification/geometrisation and the inheritance
  properties~\cite[Example~6.7]{Loeh:ell1} of $\ell^1$-invisibility.
  In this article, we prefer to stick to the case of simplicial volume
  because it is of wider interest.
\end{rem}

Recall that a degree-one map $f \colon M \to N$ induces a
monomorphism in cohomology $H^\bullet(f) \colon H^\bullet(N) \to
H^\bullet(M)$ (Remark~\ref{rem:degree1maps:epi:mono}). Then, we have
the following:

\begin{prop}\label{prop:gen:rud:conj:true:1}
  Let $\mathcal G$ be an isq-class of groups with~$\mathcal{G} \subset \Am$
  and let $f \colon M \to N$ be a
  degree-one map between oriented closed connected
  manifolds. Suppose that $N$ has $\efcat$ equal to~$k + 1$. Then,
  if $H^k_b(f)$ is injective, we have $\gcat(M) \geq \gcat (N)$.
\end{prop}

\begin{proof}[Proof of Proposition~\ref{prop:gen:rud:conj:true:1}]
  Let us consider the following commutative diagram:
  $$
  \xymatrix{
    H^k_b(N) \ar[rr]^-{H^k_b(f)} \ar[d]_-{\comp^k_N} && H^k_b(M) \ar[d]^-{\comp^k_M} \\
    H^k(N) \ar[rr]_-{H^k(f)} && H^k(M) \ . 
  }
  $$
  Since $f$ is a degree-one map, $H^\bullet(f)$ is injective
  (Remark~\ref{rem:degree1maps:epi:mono}) and the same holds true for
  $H^k_b(f)$ by assumption. Thus, the vanishing of
  $\comp^k_M$ implies the vanishing of~$\comp^k_N$. However, since by
  assumption, $\comp^k_N$ is non-trivial, we have that $\comp^k_M$ is
  also non-trivial. By Theorem~\ref{grom:van:thm}, this implies that $\gcat(M)
  \geq k + 1 = \gcat N$.
\end{proof}

We now describe sufficient conditions such that the map in the
previous result $H^k_b(f)$ is injective. First of all we see that this
is always the case in degree $2$.

\begin{cor}\label{corA:gen:rud:conj:true:1}
  Let $\mathcal{G}$ be an isq-class of groups with~$\mathcal{G} \subset \Am$. 
  If $M$ and $N$ are
  oriented closed connected manifolds with $M \geqone N$ and if $N$
  has $\efcat$ equal to~$3$, then $\gcat (M) \geq \gcat (N)$.
\end{cor}

\begin{proof}
  Let $f \colon M \to N$ be a degree-one map. In order to apply
  Proposition~\ref{prop:gen:rud:conj:true:1} we have to show that
  $H^2_b(f)$ is injective.  However, every group epimorphism induces
  an injective map on bounded cohomology of groups in degree
  $2$~\cite{Bouarich} and so this is the case for
  $H^2_b(\pi_1(f))$. Hence, by applying
  Corollary~\ref{cor:same:behaviour:pi1:spaces} we have that also
  $H^2_b(f)$ is injective. 
\end{proof}

\begin{rem}
  As mentioned in the previous section, Rudyak's conjecture for
  $\LS$-category is always true if the target manifold has $\lscat (N)
  \leq 3$. This means that the previous result should be interpreted as
  a (weaker) counterpart of the small-values-case for the generalised
  monotonicity problem.
\end{rem}

\begin{cor}\label{corB:gen:rud:conj:true:1}
  Let $\mathcal{G}$ be an isq-class of groups with~$\mathcal {G} \subset \Am$
  and let $M$ and $N$ be 
  oriented closed connected manifolds such that there exists a
  degree-one map $ M \to N$ with amenable kernel on the level of~$\pi_1$. Then, if $N$
  has $\efcat$, we have $\gcat (M) \geq \gcat (N)$.
\end{cor}
\begin{proof}
  Let $f \colon M\to N$ be such a map. 
  In order to apply Proposition~\ref{prop:gen:rud:conj:true:1}, we
  have to show that $H_b^\bullet(f)$ is injective. However,
  Theorem~\ref{map:thm} shows that under the hypothesis on~$\ker (f)$
  this is always the case.
\end{proof}

Following Bouarich~\cite{Bouarich:exact}, one can study the following
family of groups:

\begin{defi}\label{def:lambda:family}
  Let $\lex$ be the family of discrete groups~$\Lambda$ with the
  following ``left exactness'' property: For every group~$\Gamma$ and
  every group epimorphism $f \colon \Gamma \to \Lambda$, the induced
  map~$H^\bullet_b(f) \colon H^\bullet_b(\Lambda) \to
  H^\bullet_b(\Gamma)$ in bounded cohomology is injective in all
  degrees.
\end{defi} 

\begin{example}\label{example:groups:in:lambda}
  The following groups lie in~$\lex$:
  \begin{enumerate}
  \item \emph{Amenable groups}. Indeed, their bounded cohomology
    vanishes in all positive degrees~\cite{vbc}.
  \item \emph{Free groups}. Given an epimorphism $\phi \colon \Gamma
    \to F_n$ onto a free group, there exists a right inverse $\psi
    \colon F_n \to \Gamma$. Hence, the composition $$ H^\bullet_b(F_n)
    \xrightarrow{H^\bullet_b(\phi)} H^\bullet_b(\Gamma)
    \xrightarrow{H^\bullet_b(\psi)} H^\bullet_b(F_n) \
    $$
    induces the identity morphism on bounded cohomology groups in
    every dimension. This shows that $H^\bullet_b(\phi)$ is injective
    as claimed.
  \item The class~$\lex$ is stable under the following constructions: quotients
    by amenable subgroups, extension of an amenable group by an
    element of $\lex$ and free products of amenable
    groups~\cite{Bouarich:exact}.
  \item \emph{Fuchsian groups} (e.g., fundamental groups of closed
    surfaces with negative Euler characteristic)~\cite[Proposition~3.8
      and Corollary~3.9]{Bouarich:exact}.
  \item \emph{Fundamental groups of a geometric $3$-manifold}
    (e.g. Kleinian groups of finite volume). This can be shown by
    using Bouarich's result~\cite[Corollary~3.13 and
      pp.~267]{Bouarich:exact} together with Agol's proof of
    Thurston's Virtual Fibering Conjecture~\cite{Agol}.
\end{enumerate}
\end{example}

\begin{cor}\label{corC:gen:rud:conj:true:1}
  Let $\mathcal G$ be an isq-class of groups with~$\mathcal{G} \subset \Am$ 
  and let $M$ and $N$ be oriented closed
  connected manifolds with $M \geqone N$. If $N$ has $\efcat$ and
  $\pi_1(N) \in \, \lex$, then $\gcat (M) \geq \gcat (N)$.
\end{cor}

\begin{proof}
This follows from the definition of $\lex$, 
Corollary~\ref{cor:same:behaviour:pi1:spaces}
and Proposition~\ref{prop:gen:rud:conj:true:1}. 
\end{proof}

\subsection{Examples of manifolds with efficient category}

In the previous section we have described some sufficient conditions
under which the monotonicity problem has a positive answer. We now
investigate how to construct manifolds having $\efcat$, where
$\mathcal{G}$ consists of amenable groups.
Recall that orientable $\PD^n_\mathbb{R}$-groups~$\Gamma$ satisfy 
$\cd \Gamma = n$ and $H^n(\Gamma) \cong \R$.

\begin{prop}\label{prop:exists:mfld:gen:rud:hyp:groups}
  Let $\mathcal G$ be an isq-class of groups with~$\mathcal{G} \subset \Am$. 
  Let $\Gamma$ be a finitely
  presented group lying in one of
  the following families:
  \begin{enumerate}
  \item hyperbolic orientable $\PD^n_\mathbb{R}$-groups for $n \geq 3$;
  \item fundamental groups of aspherical manifolds with positive
    simplicial volume;
  \item Hopfian fundamental groups of $n$-manifolds with positive 
  simplicial volume and $\gd \Gamma = n$.
  \end{enumerate}
  Then, for every $n \geq 2 \cdot \gd \Gamma + 1$, there exists an oriented
  closed connected $n$-manifold $N$ with $\pi_1(N) \cong \Gamma$ and $\efcat$.
\end{prop}
\begin{proof}
  By Proposition~\ref{prop:catY:catmfld}.4, in all three cases it
  suffices to show that $B\Gamma$ has $\efcat$ equal to~$n+1$.
  We now consider the three cases:
  \begin{enumerate}
  \item Let $\Gamma$ be a hyperbolic orientable $\PD^n_\mathbb{R}$-group with $n \geq 3$
  and let us consider a model of $B\Gamma$.
   Then, the comparison map of $B\Gamma$ is surjective in all degrees greater
    than or equal to $2$~\cite[Theorem~3]{Mineyev}. Hence, since 
   $H^n(B\Gamma) \cong \mathbb{R}$, we have that  $\comp_{B\Gamma}^n$ is surjective and
    non-trivial. Moreover, as $n\geq 3$, we have~$\gd \Gamma = \cd \Gamma = n$. 
    This shows that $B\Gamma$ has efficient $\mathcal{G}$-category
    equal to $n+1$ (Theorem~\ref{grom:van:thm} and Remark~\ref{rem:F-cat:leq:lscat}).
  \item Let $\Gamma$ be the fundamental group of an oriented closed
    connected aspherical manifold with positive simplicial
    volume. Then, Theorem~\ref{grom:van:thm} and
    Corollary~\ref{cor:catsimvol} show that $B\Gamma$ has
    $\efcat$.
 \item Let $\Gamma$ be a Hopfian fundamental group of an oriented
  closed connected $n$-manifold $M$ with positive simplicial volume.
  Since $n = \gd \Gamma$, there exists a model of $B\Gamma$
  with $\gcat(B\Gamma) \leq n+1$ (Remark~\ref{rem:F-cat:leq:lscat}).
  Moreover, since $\sv{M} > 0$, the comparison map
   $\comp_{B\Gamma}^n$ does not vanish~\cite[Section~3.1]{vbc}.
    Hence, $B\Gamma$ has efficient $\mathcal{G}$-category
    equal to~$n+1$ (Theorem~\ref{grom:van:thm}). 
  \qedhere
\end{enumerate}
\end{proof}


\subsection{Examples of manifolds that satisfy monotonicity}

We are now ready to produce explicit infinite families of examples of
target manifolds for which monotonicity holds. Recall that
$\lex$ denotes the family of groups defined in
Definition~\ref{def:lambda:family}.

\begin{thm}\label{thm:main:gen:rud:conj}
  Let $\mathcal{G}$ be an isq-class of groups with~$\mathcal{G} \subset \Am$.
  Let $\Gamma \in \lex$ be a
  finitely presented group and suppose that $\Gamma$ lies in one of the
  families of Proposition~\ref{prop:exists:mfld:gen:rud:hyp:groups}.
  Then for every $n \geq 2 \cdot \gd \Gamma + 1$, there exists an oriented
  closed connected $n$-manifold $N$ with $\pi_1(N) \cong \Gamma$ such
  that for all oriented closed connected $n$-manifolds~$M$ we have
  \[ M \geqone N \Longrightarrow \gcat M \geq \gcat N\ . 
  \]
\end{thm}

\begin{proof}
  Since $\Gamma$ lies in the families of
  Proposition~\ref{prop:exists:mfld:gen:rud:hyp:groups}, we have that
  for every $n \geq 2 \cdot \gd \Gamma + 1$, there exists an oriented
  closed connected $n$-manifold~$N$ that has $\efcat$. Because $\Gamma
  \in \, \lex$ we can apply Corollary~\ref{corC:gen:rud:conj:true:1}.
\end{proof}

As an application of the previous result, we can prove Theorem~\ref{thm:gen:rudyak:conj:intro}.

\begin{cor}\label{thm:gen:rudyak:conj:intro:new}
  Let $\mathcal{G}$ be an isq-class of groups with~$\mathcal{G} \subset \Am$. 
  Let $\Gamma$ be the fundamental group of
  an oriented closed connected hyperbolic $k$-manifold of dimension~$k
  \in \{2,3\}$. Then, for
  every~$n \geq 2 k$ there exists an oriented closed connected
  $n$-manifold~$N$ with $\pi_1(N) \cong \Gamma$ such that: For all
  oriented closed connected $n$-manifolds~$M$ we have
  \[ M \geqone N \Longrightarrow \gcat M \geq \gcat N.\]
\end{cor}

\begin{proof}
  Since $\Gamma$ is the fundamental group of an oriented closed
  connected aspherical $n$-manifold with positive simplicial volume, 
  $\Gamma$ lies in the classes of
  Proposition~\ref{prop:exists:mfld:gen:rud:hyp:groups} and we can use
  the improved dimension bound
  (Remark~\ref{rem:improveddim}). Moreover, we have already seen in
  Example~\ref{example:groups:in:lambda} that $\Gamma$ also lies
  in~$\lex$. Hence, Theorem~\ref{thm:main:gen:rud:conj} applies.
\end{proof}

\section{Bounds via classifying spaces}\label{sec:catclass}

In this section, we explain how to
characterise categorical invariants for classes of groups in terms of
classifying spaces of families of subgroups -- in analogy with the
corresponding statement for~$\catop_1$ by Eilenberg and
Ganea~\cite{eilenbergganea}. In particular, this leads to
a corresponding lower bound in terms of Bredon cohomology.

\subsection{Classifying spaces of families}

We recall classifying spaces of families of subgroups~\cite{Lueck:survey}.

\begin{defi}
	\label{defn:classifying:spaces}
	Let $\Gamma$ be a group and let $F$ be a subgroup family of $\Gamma$.
	\begin{itemize}
		\setlength\itemsep{0em}
		\item A $\Gamma$-CW-complex $X$ has
                  \emph{$F$-restricted isotropy} if all its isotropy
                  groups lie in $F$, i.e., for every $x\in X$ one has
		\[
		\Gamma_x=\{\gamma\in \Gamma \;|\; \gamma\cdot x =x\}\in F.
		\]
		\item A \emph{model for the classifying space
                  $E_F\Gamma$ for the family of subgroups $F$} is a
                  $\Gamma$-CW-complex $X$ with $F$-restricted isotropy
                  with the following universal property: for every
                  $\Gamma$-CW-complex~$Y$ with $F$-restricted isotropy
                  there exists a unique (up to $\Gamma$-homotopy)
                  $\Gamma$-equivariant map~$Y\rightarrow X$.
	\end{itemize}
	We will use the notation $E_F\Gamma$ to denote the
        choice of a model for the classifying space (it is well
        defined up to canonical $\Gamma$-homotopy equivalence) and we
        denote by~$f_{Y,\Gamma,F}:Y\rightarrow E_F\Gamma$ a choice of a
        map given by the universal property.
       
        When $F$ is the trivial family, then $E_F\Gamma$ is denoted
        simply by $E\Gamma$ and we recover the usual model of the classifying
        space of a group $\Gamma$.
\end{defi}

\begin{rem}\label{rem:admissibility}
	When $\Gamma$ is a discrete group, a
        $\Gamma$-CW-complex~$X$ is simply a CW-complex~$X$ endowed
        with a cellular $\Gamma$-action with the following property: For each
        open cell $e \subset X$ and each $\gamma \in \, \Gamma$ such
        that $\gamma \cdot e \cap e \neq \emptyset$, the left
        multiplication by $\gamma$ restricts to the identity on
        $e$~\cite[Example~1.5]{Lueck:survey}.
	
	This property shows that if a discrete group~$\Gamma$ acts on
        a simplicial complex~$K$ via simplicial automorphisms, it is
        not true in general that $K$ has the structure of
        $\Gamma$-CW-complex. On the other hand, the induced action
        by~$\Gamma$ on the first barycentric subdivision~$K'$ of $K$
        makes $K'$ a $\Gamma$-CW-complex.
\end{rem}

The remark above suggests the following definition:

\begin{defi}
	Let $K$ be a simplicial complex and let $\Gamma$ be a discrete
        group acting on $K$ via simplicial automorphisms.  We say that
        $K$ is an \emph{admissible} $\Gamma$-\emph{simplicial complex}
        if it is a $\Gamma$-CW-complex.
\end{defi}

Classifying spaces for subgroup families have recently been used to
give a new proof of Gromov's Vanishing Theorem
(Theorem~\ref{grom:van:thm})~\cite{LS}. We will use some results of
\emph{loc.\ cit.}  to prove
Lemma~\ref{lem:cat:via:factorization:classifying:map}.

\subsection{Bounding topological complexity via classifying spaces}

Classifying spaces for subgroup families also play an important role in the
computation of the topological complexity of aspherical spaces. The
following result shows that topological complexity can be
characterised by means of the factorisation of the classifying map
$f_{E(\Gamma\times \Gamma),\Gamma\times \Gamma,\mathcal{D}}$, where
$\mathcal{D}$ is the subgroup family introduced in
Definition~\ref{def:D:TC}.

\begin{prop}[{\cite[Theorem 3.3]{FGLO}}]
	\label{prop:TC:via:factorization:classifying:map}
	Let $X$ be a connected finite aspherical CW-complex with
        fundamental group $\Gamma$.  Then $\TC(X)-1$
        coincides with the minimal integer $k\in \mathbb{N}_{\geq 0}$
        such that the classifying map
	\[
	f_{E(\Gamma\times \Gamma),\Gamma\times \Gamma,\mathcal{D}}\colon E(\Gamma \times \Gamma)\rightarrow E_{\mathcal{D}}(\Gamma \times \Gamma)
	\]
	is $(\Gamma\times \Gamma)$-homotopic to a map with values in
        the $k$-dimensional skeleton~$E_{\mathcal{D}}(\Gamma\times
        \Gamma)^{(k)}$ (of any model of $E_{\mathcal{D}}(\Gamma \times \Gamma)$).
\end{prop}

Using the previous result, Dranishnikov computed the topological
complexity of finitely generated geometrically finite hyperbolic
groups~\cite{Dran:hyp:groups}.

\subsection{Bounding category via classifying spaces}

Similarly, we can also characterise $F$-categories:

\begin{prop}\label{prop:introgencat:new}
  Let $X$ be a connected CW-complex and let $\widetilde{X}$ be 
  its universal covering. Let $\Gamma$ be the fundamental
  group of~$X$ and let $F$ be a subgroup family of~$\Gamma$.  Then,
  $\catop_F(X)-1$ coincides with the minimal integer
  $k\in\mathbb{N}_{\geq 0}$ such that the classifying map
  \[
  f_{\widetilde X,\Gamma, F}\colon\widetilde X \rightarrow E_F\Gamma
  \]
  is $\Gamma$-homotopic to a map with values in the
  $k$-dimensional skeleton~$E_F\Gamma^{(k)}$ (of any model of~$E_F\Gamma$).
\end{prop}

The proof of Proposition is based on the following lemma.

\begin{lemma}
	\label{lem:cat:via:factorization:classifying:map}
	Let $X$ be a connected CW-complex and let $n\in
        \mathbb{N}$. Let $F$ be a subgroup family of
        the fundamental group~$\Gamma$ of~$X$. Then, $\catop_F(X)\leq n+1 $ if and only
        if there exists a connected $\Gamma$-CW-complex~$L$ with
        $F$-restricted isotropy of dimension at most~$n$ and a
        $\Gamma$-map~$\widetilde X \rightarrow L$.
\end{lemma}
\begin{proof}
	If we assume that $\catop_F(X)\leq n+1$, then there exists an
        open $F$-cover~$U$ of~$X$ of cardinality at most~$n+1$. Let $p
        \colon \widetilde X \longrightarrow X$ denote the universal
        covering map and let $\widetilde U$ be the corresponding open
        cover of $\widetilde X$ (as in the proof of
        Proposition~\ref{prop:pi1amcat2}):
	\begin{align*}
		\widetilde U
		:= \bigl\{ V \subset \widetilde X
		\bigm|
		& \;\text{there exists a~$W \in U$ such that}
		\\[-.5ex]
		& \;\text{$V$ is a path-connected component of~$p^{-1}(W)$}.
		\bigr\}
	\end{align*}
 	Now we take $L$ to be the nerve of~$\widetilde U$ and $f\colon
        \widetilde X\rightarrow |L|$ to be a nerve map. Then, $|L|$ is a
        $\Gamma$-CW-complex~\cite[Lemma~4.5]{LS}. Moreover, the
        nerve map is a $\Gamma$-map~{\cite[Lemma 4.8]{LS}} and
        clearly we have $\dim(|L|)\leq \dim(N)\leq n$, where $N$ denotes
        the nerve of~$U$. Finally, $|L|$ has $F$-restricted
        isotropy~{\cite[Lemma~4.11]{LS}}.
 	
 	Conversely, let $L$ be a $\Gamma$-CW-complex with
        $F$-restricted isotropy of dimension at most~$n$ and let
        $f\colon \widetilde X \rightarrow L$ be a $\Gamma$-map.  By
        the equivariant version of simplicial
        approximation~{\cite[Proposition~A.4]{OS}}, we can assume
        without loss of generality that $L$ is a connected admissible
        $\Gamma$-simplicial complex with $F$-restricted isotropy and
        with dimension at most~$n$. Let $U$ be the open cover of~$L'$
        consisting of the open stars of the barycentric
        subdivision~$L'$ of~$L$, indexed by the dimension of
        the underlying simplices of~$L$. Grouping the sets of~$U$
        according to the dimension of the underlying simplices,
        we can replace~$U$ by an open cover of $n+1$~open sets.
        Let us pull back~$U$ via~$f$ and push
        it down to~$X$ via~$p$. We get the following open cover of~$X$:
 	\[
 	\bigl\{p(f^{-1}(V))
        \bigm| V\in U \bigr\}.
 	\]
 	By construction, the cardinality of this open cover is again at
        most~$n+1$.

        Hence, to conclude it is enough to show that this cover is
        indeed an $F$-cover of~$X$.  Let $V\in U$. We show now that
        for every~$x\in f^{-1}(V)$ there exists a~$y\in L$ such that
 	\[
 	\im\bigl(\pi_1(p(f^{-1}(V))\hookrightarrow X, p(x))\bigr)
 	\]
 	is a subgroup (up to conjugation) of the isotropy group
        $\Gamma_y$. The result then easily follows since $L$ has
        $F$-restricted isotropy and $F$ is closed under taking
        subgroups.

        So, let $x\in f^{-1}(V)$. Since we need to prove
        the statement ``up to conjugation'', we can assume that
        $\Gamma=\pi_1(X,x)$. Let $\gamma \in
        \im(\pi_1(p(f^{-1}(V))\hookrightarrow X,p(x)))$,
        i.e., $\gamma=[\sigma]$ where $\sigma \colon [0,1]\rightarrow
        p(f^{-1}(V))\subseteq X$ is a loop based at~$p(x)$.  By
        the lifting properties of the covering
 	\[
 	p_{|f^{-1}(V)}:f^{-1}(V)\rightarrow p(f^{-1}(V)),
 	\]
 	there exists a lift~$\widetilde \sigma \colon [0,1]\rightarrow
        f^{-1}(V)\subseteq \widetilde X$ such that
        $p\circ\widetilde\sigma =\sigma$ and $\widetilde\sigma(0)=x$.
        Moreover, by definition of the deck transformation action and by uniqueness
        of the lift, we have that $\gamma\cdot
        \widetilde\sigma(1)=x$.  Now, since $f$ is a $\Gamma$-map, we
        also have
 	\[
 	\gamma\cdot f(\widetilde\sigma(1))=f(x).
 	\]
 	Moreover, both $f(\widetilde\sigma(1))$ and $f(x)$ are in $V$
 	 and $f\circ \widetilde\sigma$ is a path connecting them
        contained in~$V$. Hence they lie in the same path connected
        component of~$V$ and, by construction of the cover~$U$, there
        exists a vertex $v$ of $L'$ such that
        $f(\widetilde\sigma(1)),f(x)\in \st(v)$.  Since the action on
        $L'$ is simplicial, we know that
        $\gamma\cdot\st(v)=\st(\gamma\cdot v)$.  It follows that
        $\st(v)\cap \st(\gamma\cdot v)\neq \emptyset$.  Hence, by
        definition of open stars, we have that there exists a
        simplex~$\sigma \subset L'$ such that $\gamma \cdot
        \textup{int}(\sigma) \cap \textup{int}(\sigma) \neq \emptyset$
        and whose set of vertices contains both $v$ and $\gamma \cdot
        v$. Using the fact that $L'$ is an admissible
        $\Gamma$-simplicial complex, Remark~\ref{rem:admissibility}
        readily implies that $\gamma \cdot \sigma = \sigma$. This
        shows that $\gamma \cdot v = v$, whence $\gamma \in \,
        \Gamma_v$, as desired.
\end{proof}

\begin{proof}[Proof of Proposition~\ref{prop:introgencat}]
	The result follows from
        Lemma~\ref{lem:cat:via:factorization:classifying:map} by using
        the universal property of classifying spaces for subgroup
        families and the equivariant cellular approximation
        theorem~{\cite[Section 1.1]{Lueck:survey}}.
\end{proof}

The previous proposition has implications in terms of Bredon
cohomology, as shown by Farber, Grant, Lupton, and
Oprea~\cite{FGLO}. Namely, it provides a lower bound for the
topological complexity, via the diagonal category, in terms of a
vanishing map in Bredon cohomology~{\cite[Theorem 4.1]{FGLO}}. Thanks
to Proposition~\ref{prop:introgencat:new}, this argument also applies to
general subgroup families.

\begin{cor}
	\label{cor:lower:bound:bredon}
	Let $X$ be an aspherical connected CW-complex. Let $\Gamma$ be the fundamental group of $X$ and let $F$ be a subgroup family of~$\Gamma$.
	Suppose that there exists~$k\in \mathbb{N}_{\geq0}$ and a Bredon module~$M$ such that the restriction
	\[
	H^k_F(\Gamma;M)\rightarrow H^k(\Gamma;\textup{res}_{\{1\}}^F M)
	\]
	is non-zero. Then $\catop_F(X)\geq k+1$.
\end{cor}
\begin{proof}
  Using the factorisation characterisation of~$\catop_F$
  (Proposition~\ref{prop:introgencat:new}), we can proceed 
  as in the proof for the diagonal family~\cite[Theorem~4.1]{FGLO}.
\end{proof}  

The following application is due to Kevin Li:

\begin{example}\label{exa:hypbredon}
  Let $M$ be an oriented closed connected aspherical $n$-manifold
  whose fundamental group is hyperbolic and let $n\geq 2$.
  Then $\amcat M = n+1$. 
  Of course, this follows from Mineyev's result on the comparison map
  in bounded cohomology and Gromov's vanishing theorem (see the proof
  of Proposition~\ref{prop:exists:mfld:gen:rud:hyp:groups}).
  Using Corollary~\ref{cor:lower:bound:bredon}, one can give an alternative
  argument: 
  Because $\Gamma$ is hyperbolic (and torsion-free), all amenable
  subgroups of~$\Gamma$ are virtually cyclic and the construction
  principle of Juan-Pineda and Leary can be used to obtain a model
  of~$E_{\Am}\Gamma$ from a model of~$E \Gamma$~\cite[Remark~7]{jpl}.
  Therefore, the restriction map
  \[ H^k_{\Am}(\Gamma;\R) \longrightarrow H^k(\Gamma;\R) 
  \]
  is an epimorphism for every~$k \in \N_{\geq 2}$ (this is the
  cohomological analogue of the exact sequence of Juan-Pineda and
  Leary~\cite[Proposition 18]{jpl}). 
  Moreover,
  $M$ is a model of~$B\Gamma$ and so~$H^n(\Gamma;\R) \cong H^n(M)
  \cong\R$.  So, Corollary~\ref{cor:lower:bound:bredon} shows
  that $\amcat M \geq n+1$.
\end{example}

\section{Amenable category vs. topological complexity}

We now investigate the following question:

\begin{quest}\label{q:amtc:new}
  For which topological spaces~$X$ do we have
  \[ \amcat(X \times X) \leq \TC(X) \text{\;?}
  \]
\end{quest} 

\subsection{Basic examples}

As warm-up examples behind this question, we consider spaces with
amenable fundamental group, wedges of circles, and surfaces.

\begin{example}\label{exa:amtcamenable}
  As the product of two amenable groups is amenable, we have $\amcat
  (X \times X) = 1$ for all spaces~$X$ with amenable fundamental
  group. Therefore, the property in Question~\ref{q:amtc:new} holds in
  this case.
\end{example}

\begin{example}\label{exa:amtcwedge}
  Let $X = \bigvee_I \mathbb{S}^1$ be a finite wedge of circles. Then we have
  $\TC(X)\geq 3$~{\cite[Theorem~7.3]{Farber:instabilities}}. 
  Hence, by applying the product formula
  (Proposition~\ref{prop:cat:products}) we get
  \[ \amcat(X\times X)\leq 2\cdot \catop_{\Am}(X) - 1 \leq 3 \ ,\]
  whence the inequality in Question~\ref{q:amtc:new}.
\end{example}

\begin{lemma}
  If $S$ is an oriented connected surface, then we have
  $$
  \amcat (S \times S) \leq \TC(S) \ .
  $$
\end{lemma}
\begin{proof}
  Let us assume first that $S$ is a non-compact surface or a surface
  with boundary. Then, we
  know that $S$ has (possibly trivial) free fundamental group and it
  retracts to a (possibly trivial) wedge of circles. Hence, since both
  $\Am$-category and $\TC$ are homotopy invariants, the desired
  inequality comes from the computation in Example~\ref{exa:amtcwedge}.

  Let us now assume that $S$ is a closed connected surface. If $S$ has
  amenable fundamental group, then we are done by
  Example~\ref{exa:amtcamenable}.

  So we are reduced to consider the case of hyperbolic surfaces
  $\Sigma_g$, with $g \geq 2$. On the one hand, we know that
  $\TC(\Sigma_g)=5$~{\cite[Theorem
      9]{Farber:top:compl:of:motion:planning}}. On the other hand, by
  Example~\ref{exa:possimvol} and Corollary~\ref{cor:catsimvol}, we
  have $\amcat (\Sigma_g \times \Sigma_g) = 5$. This shows that 
  $\amcat (S \times S) \leq \TC(S)$ is verified also in
  this case.
\end{proof}

We show in the next section that the results about wedges of
circles and surfaces admit natural generalisations.

\subsection{Proof of Theorem~\ref{thm:examples}}\label{sec:examples}

In this section, we will prove the following:

\begin{thm}\label{thm:examples:new}
  The following classes of spaces satisfy the estimate in
  Question~\ref{q:amtc:new}:
  \begin{enumerate}
  \item Spaces with amenable fundamental group;
  \item Spaces of type $B\Gamma$ where $\Gamma$ is a finitely
    generated geometrically finite hyperbolic group;
  \item Spaces of type $B\Gamma$ where $\Gamma=H*H$ is the free square
    of a geometrically finite group $H$;
  \item Manifolds whose fundamental group is the fundamental group of
    a graph of groups whose vertex (and edge) groups are all amenable.
  \end{enumerate}
\end{thm}

\begin{proof}
  \emph{Ad~1.} We have already discussed this case in
  Example~\ref{exa:amtcamenable}.

  \emph{Ad~2.} Let $X$ be a model of~$B\Gamma$, where $\Gamma$ is
  a finitely generated geometrically finite hyperbolic group. Then, we
  have that $2\cdot\gd\Gamma + 1 \leq
  \TC(X)$~\cite[Theorem~3.0.2]{Dran:hyp:groups}. Hence, we get the
  following chain of inequalities:
  \begin{align*}
    \catop_{\Am}(X \times X) & \leq  \lscat(X \times X)  & \\
    &  \leq 2\cdot\gd\Gamma + 1 & \\
    & = \TC(X) \ .
  \end{align*}

  \emph{Ad~3.} Let $H$ be a geometrically finite group and let $\Gamma
  = H * H$. Assume that $X$ and $Y$ are model of $B\Gamma$ and $BH$ of
  minimal dimension, respectively. Then, $X$ is homotopy equivalent to
  $Y\lor Y$. This shows the following
  \begin{align*}
    \TC(X) & =\TC(Y\lor Y)  & \\
    & =\max\{  \TC(Y), 2\cdot \gd H+1 \} & \text{ by {\cite[\text{Theorem 2}]{Dran:Sadykov:free:product}}}\\
    & =\max\{ \TC(Y), 2\cdot\gd\Gamma +1 \} & \\
    & \geq 2 \cdot \gd\Gamma + 1 & \\
    & \geq \lscat(X\times X) & \\
    & \geq \amcat(X\times X).
  \end{align*}

  \emph{Ad~4.}
  Let $M$ be a manifold whose fundamental group~$\Gamma$ is the
  fundamental group of a graph of groups whose vertex (and edge)
  groups are all amenable.  If $\Gamma$ is amenable, then we
  are in the case of the first part.

  So, let us assume that $\Gamma$ is non-amenable. Then, there exists
  a model~$X$ of~$B\Gamma$ with $\amcat(X) = 2$
  (Theorem~\ref{thm:smallamcat}). Hence, given a classifying map $f
  \colon M \to X$ for $M$, by Remark~\ref{rem:pullbackcat} we get
  $\amcat(M) \leq 2$. This implies that $\amcat(X \times X)
  \leq 3$ by the product formula
  (Proposition~\ref{prop:cat:products}).
      
  If $\lscat (M) \leq 2$, then $M$ is either contractible or a homotopy sphere~\cite[page~336]{James}. 
  This shows that $M$ has abelian (whence amenable)
  fundamental group. As we assumed $\Gamma$ to be
  non-amenable, we have $\lscat (M) \geq 3$ and thus obtain the
  following chain of inequalities:
  \[
  \amcat(M \times M) \leq 3 \leq \lscat (M) \leq \TC(M) \ .
  \qedhere
  \] 
\end{proof}

\begin{rem}
	\label{rem:lscat:leq:tc:does:not:hold}
	In the parts (2) and (3) of Theorem~\ref{thm:examples} we actually showed that $\lscat(X \times X)\leq \TC(X)$ holds for the LS-category. However, this does
        not hold in general: Indeed, if we consider the $n$-dimensional torus~$T^n$ for~$n\in \mathbb{N}_{\geq 2}$, we have $\lscat(T^n\times T^n)=2n+1$, but $\TC(T^n)=n+1$~\cite{Farber:top:compl:of:motion:planning}.
\end{rem}

\subsection{Manifolds with $3$-manifold fundamental group}

We show that the techniques introduced in
Section~\ref{sec:cplxs:mflds} allow us to construct infinite families
of manifolds for whom Question~\ref{q:amtc:new} is answered positively. In fact, the following
results are also true for all isq-classes of groups~$\mathcal{G}
\subset \Am$.

\begin{thm}\label{thm:strictly:ess:mfld:am:TC}
  Let $\mathcal{G}$ be an isq-class of groups $\mathcal{G} \subset \Am$. 
 Let $Y$ be an oriented closed connected triangulable $n$-manifold
   whose fundamental group is Hopfian. 
  Let $\overline{k} > 1$ be the maximal integer such that $Y$ is essential
  in degree~$\overline{k}$. Then, if $$\gcat (Y) \leq \frac{\overline{k}}{2} +
  \frac{3}{2},$$ there exists for every $n \geq 2 \dim(Y)$ an
  oriented closed connected $n$-manifold~$M$ such that $\pi_1(M) \cong
  \pi_1(Y)$ and
  $$
  \gcat (M \times M) \leq \TC(M).
  $$
\end{thm}

\begin{proof}
  For every $n \geq 2 \cdot \dim(Y)$, by
  Corollary~\ref{cor:hopfian:strictly:essential}, there exists an oriented closed connected
  triangulable manifold~$M$ that is essential in degree~$\overline{k}$ with 
  $$
  \gcat (M) \leq \gcat (Y)  \qand \pi_1(M) \cong \pi_1(Y).
  $$
  Hence, using Proposition~\ref{prop:cat:products} and the hypothesis on~$\gcat (Y)$,
  we obtain
  $$
  \gcat (M \times M) \leq 2 \cdot \gcat(M)  -1 \leq 2 \cdot \gcat(Y) -1 \leq \overline{k} + 2.
  $$

  By Remark~\ref{rem:essential:in:degree:vs:strictly} $M$ is also a strictly 
  $\overline{k}$-essential manifold. Hence, it satisfies $\lscat (M) \geq
  \overline{k} + 2$~\cite[Theorem~5.2]{DKR}. Since we already know
  that $\lscat (M) \leq \TC(M)$ by Proposition~\ref{prop:bounds:TC},
  we get
  \begin{align*}
    \gcat (M \times M)
    & \leq 2 \gcat(M)  -1 \leq 2 \gcat(Y) -1 \leq \overline{k} + 2 \leq \lscat(M)
    \\
    & \leq \TC(M).
    \qedhere
  \end{align*}
\end{proof}

As an application of the previous result, we obtain Theorem~\ref{thm:am:vs:tc:3:mflds:intro}.

\begin{cor}\label{thm:am:vs:tc:3:mflds:intro:new}
  Let $\mathcal{G}$ be an isq-class of groups with $\mathcal{G} \subset \Am$.
  Let $Y$ be an oriented
  closed connected $3$-manifold which is a connected sum of graph
  manifolds. Then, for every $n \geq 6$, there exists an oriented
  closed connected $n$-manifold $M$ with $\pi_1(M) \cong \pi_1(Y)$ and
  such that
  $$
  \gcat(M \times M) \leq \TC (M) \ .
  $$  
\end{cor}

\begin{proof}
  Since $Y$ is an oriented closed connected $3$-manifold, its fundamental
  group is residually finite~\cite{Hempel}, whence Hopfian.
  Moreover, $Y$ being a connected
  sum of aspherical manifolds, is essential in degee~$3$.
  Finally, $\gcat (Y) \leq 3$
  because $Y$ is the connected sum of graph
  manifolds~\cite[Theorem~2]{GGH:am}\cite[Corollary~5]{GGW}. Hence, we have
  $$
  \gcat (Y) \leq 3 = \frac{3}{2} + \frac{3}{2} = \frac{\dim(Y)}{2} + \frac{3}{2}.
  $$
  This inequality shows that we are in the situation of
  Theorem~\ref{thm:strictly:ess:mfld:am:TC}, whence we get the thesis.
\end{proof}

It is a natural question to ask if the previous result can be improved by
considering $M$ to be actually a connected sum of graph manifolds.
As far as we know, our approach does not lead to this stronger statement, 
but we refer the reader to recent results by Mescher~\cite{mescher} in this direction.

 
\bibliographystyle{abbrv}
\bibliography{svbib}

\end{document}